\documentclass[a4paper]{amsart}

\usepackage{a4,exscale,amsfonts,amssymb,amsmath,amscd}
\usepackage{nicefrac,tikz,tikz-3dplot,graphicx,color}
\usetikzlibrary{arrows,calc,positioning}
\usepackage[font=scriptsize]{caption}
\usepackage[english]{babel}

\allowdisplaybreaks

\usepackage[mathscr]{eucal}
\usepackage{color}

\newenvironment{proofof}[1]{{\noindent\bf Proof of #1 }}{\hfill$\square$}

\newcommand{\Abb}[5]{\begin{array}{ccccc}#1&:&#2&\longrightarrow&#3\\{}&{}&#4&\longmapsto&#5\end{array}}

\newcommand{\ol}{\overline}

\newcommand{\nz}{\mathbb{N}}

\newcommand{\rz}{\mathbb{R}}

\newcommand{\rt}{\rz^3}

\DeclareMathOperator{\p}{\partial}

\newcommand{\na}{\nabla}

\DeclareMathOperator{\Grad}{Grad}

\DeclareMathOperator{\rot}{rot}
\DeclareMathOperator{\Rot}{Rot}

\renewcommand{\div}{\operatorname{div}}

\DeclareMathOperator{\Div}{Div}
\newcommand{\T}{\mathcal{T}}
\newcommand{\E}{\mathcal{E}}
\newcommand{\B}{\mathcal{B}}
\newcommand{\A}{\mathcal{A}}

\newcommand{\eps}{\varepsilon}

\newcommand{\om}{\Omega}

\newcommand{\equi}{\Leftrightarrow}
\def\impl{\Rightarrow}

\newcommand{\qequi}{\quad\equi\quad}

\newcommand{\qimpl}{\quad\impl\quad}

\DeclareMathOperator{\supp}{supp}

\DeclareMathOperator{\dist}{dist}

\newcommand{\dvec}[3]{\begin{bmatrix}#1\\#2\\#3\end{bmatrix}}

\DeclareMathOperator{\id}{id}

\DeclareMathOperator{\Lebesgue}{\mathsf{L}}
\newcommand{\Lgen}[2]{\Lebesgue^{#1}_{#2}}
\def\Li{\Lgen{\infty}{}}
\def\Liom{\Li(\om)}
\def\Lo{\Lgen{1}{}}

\def\Lt{\Lgen{2}{}}
\def\Ltperp{\Lgen{2}{\bot}}
\def\Ltom{\Lt(\om)}
\def\LtXi{\Lt(\Xi)}
\def\Ltperpom{\Ltperp(\om)}

\def\Ltepsom{\Lgen{2}{\eps}(\om)}

\def\L6{\Lgen{6}{}}

\DeclareMathOperator{\Sobolev}{\mathsf{H}}
\newcommand{\Hgen}[3]{\overset{#3}{\Sobolev}{}^{#1}_{#2}}
\newcommand{\Hgentr}[3]{\overset{#3}{\utilde{\Sobolev}}{}^{#1}_{#2}}

\def\Ho{\Hgen{1}{}{}}
\def\Ht{\Hgen{2}{}{}}

\def\Hoom{\Ho(\om)}
\def\Hoperp{\Hgen{1}{\perp}{}}
\def\Hoperpom{\Hgen{1}{\perp}{}(\om)}

\def\Hoc{\Hgen{1}{}{\circ}}

\def\Hocom{\Hoc(\om)}

\def\Hoctr{\Hgentr{1}{}{\circ}}
\def\Hoct{\Hgen{1}{\Gamma_\tau}{\circ}}
\def\Hocttr{\Hgentr{1}{\Gamma_\tau}{\circ}}
\def\Hocn{\Hgen{1}{\Gamma_\nu}{\circ}}
\def\Hocgn{\Hgen{1}{\gamma_\nu}{\circ}}
\def\Hocntr{\Hgentr{1}{\Gamma_\nu}{\circ}}
\def\Hocgntr{\Hgentr{1}{\gamma_\nu}{\circ}}
\def\Hocgt{\Hgen{1}{\gamma_\tau}{\circ}}

\def\Hoctom{\Hoct(\om)}
\def\Hocnom{\Hocn(\om)}

\def\Hocgng{\Hocgn(\Xi)}

\def\HoctUT{\Hgen{1}{\Upsilon_0}{\circ}(\Theta)}
\def\HoctUTt{\Hgen{1}{\tilde\Upsilon_0}{\circ}(\tilde\Theta)}

\DeclareMathOperator{\CSobolev}{\mathcal{H}}
\newcommand{\cHgen}[3]{\overset{#3}{\CSobolev}{}^{#1}_{#2}}

\def\cHoc{\cHgen{1}{}{\circ}}

\def\cHocom{\cHoc(\om)}

\def\cHoct{\cHgen{1}{\Gamma_\tau}{\circ}}

\def\cHoctom{\cHoct(\om)}

\DeclareMathOperator{\Cont}{\mathsf{C}}
\newcommand{\Cgen}[2]{\overset{#2}{\Cont}{}^{#1}}

\def\Czo{\Cgen{0,1}{}}
\def\Czoc{\Cgen{0,1}{\circ}}
\def\Czoct{\Cgen{0,1}{\circ}_{\Gamma_\tau}}
\def\Czocn{\Cgen{0,1}{\circ}_{\Gamma_\nu}}

\def\Ci{\Cgen{\infty}{}}
\def\Cic{\Cgen{\infty}{\circ}}
\def\Cicrt{\Cic(\rt)}
\def\Cicn{\Cic_{\Gamma_\nu}}
\def\Cict{\Cic_{\Gamma_\tau}}

\def\Cictom{\Cic_{\Gamma_\tau}(\om)}
\def\Cicnom{\Cic_{\Gamma_\nu}(\om)}
\def\Cictg{\Cic_{\gamma_\tau}(\Xi)}

\DeclareMathOperator{\dirichlet}{\mathcal{H}}
\newcommand{\qharmdi}[2]{\dirichlet^{#1}_{#2}(\om)}

\newcommand{\harmdieps}{\qharmdi{}{\eps}}

\newcommand{\normdst}{\hspace{-0.4ex}}
\newcommand{\scp}[2]{\langle#1,#2\rangle}

\newcommand{\scpLtom}[2]{\scp{#1}{#2}_{\Ltom}}

\newcommand{\scpLtXi}[2]{\scp{#1}{#2}_{\LtXi}}
\newcommand{\scpLtXit}[2]{\scp{#1}{#2}_{\Lt(\tilde\Xi)}}

\newcommand{\norm}[1]{\left|\normdst\left|#1\right|\normdst\right|}
\newcommand{\norms}[1]{|\normdst|#1|\normdst|}

\newcommand{\normsLtom}[1]{\norms{#1}_{\Ltom}}

\newcommand{\normsLtg}[1]{\norms{#1}_{\LtXi}}

\newtheorem{lem}{Lemma}[section]
\newtheorem{defi}[lem]{Definition}
\newtheorem{theo}[lem]{Theorem}

\newtheorem{rem}[lem]{Remark}

\DeclareMathOperator{\rotspace}{\mathsf{R}}
\newcommand{\rotgen}[2]{\overset{#2}{\rotspace}{}_{#1}}
\newcommand{\rotgenom}[2]{\rotgen{#1}{#2}(\om)}
\newcommand{\rotgeng}[2]{\rotgen{#1}{#2}(\Xi)}
\newcommand{\rom}{\rotgenom{}{}}

\newcommand{\rcom}{\rotgenom{}{\circ}}
\newcommand{\rzom}{\rotgenom{0}{}}

\newcommand{\rzg}{\rotgen{0}{}(\Xi)}

\newcommand{\rcn}{\rotgen{\Gamma_\nu}{\circ}}
\newcommand{\rct}{\rotgen{\Gamma_\tau}{\circ}}

\newcommand{\rcng}{\rotgeng{\gamma_\nu}{\circ}}
\newcommand{\rctg}{\rotgeng{\gamma_\tau}{\circ}}
\newcommand{\rczng}{\rotgeng{\gamma_\nu,0}{\circ}}
\newcommand{\rcztg}{\rotgeng{\gamma_\tau,0}{\circ}}

\newcommand{\rcthat}{\rotgen{\hat\Gamma_{\tau,k}}{\circ}}

\newcommand{\rcnom}{\rotgenom{\Gamma_\nu}{\circ}}
\newcommand{\rctom}{\rotgenom{\Gamma_\tau}{\circ}}
\newcommand{\rcznom}{\rotgenom{\Gamma_\nu,0}{\circ}}
\newcommand{\rcztom}{\rotgenom{\Gamma_\tau,0}{\circ}}

\DeclareMathOperator{\crotspace}{\mathcal{R}}
\newcommand{\crotgen}[2]{\overset{#2}{\crotspace}{}_{#1}}
\newcommand{\crotgenom}[2]{\crotgen{#1}{#2}(\om)}
\newcommand{\crotgeng}[2]{\crotgen{#1}{#2}(\Xi)}

\newcommand{\crcom}{\crotgen{}{\circ}(\om)}

\newcommand{\crcn}{\crotgen{\Gamma_\nu}{\circ}}
\newcommand{\crcthat}{\crotgen{\hat\Gamma_{\tau,k}}{\circ}}

\newcommand{\crcng}{\crotgeng{\gamma_\nu}{\circ}}
\newcommand{\crctg}{\crotgeng{\gamma_\tau}{\circ}}

\newcommand{\crczng}{\crotgeng{\gamma_\nu,0}{\circ}}

\newcommand{\crctom}{\crotgenom{\Gamma_\tau}{\circ}}
\newcommand{\crcnom}{\crotgenom{\Gamma_\nu}{\circ}}
\newcommand{\crcztom}{\crotgenom{\Gamma_\tau,0}{\circ}}
\newcommand{\crcznom}{\crotgenom{\Gamma_\nu,0}{\circ}}

\DeclareMathOperator{\divspace}{\mathsf{D}}
\newcommand{\divgen}[2]{\overset{#2}{\divspace}{}_{#1}}
\newcommand{\divgenom}[2]{\divgen{#1}{#2}(\om)}
\newcommand{\divgeng}[2]{\divgen{#1}{#2}(\Xi)}
\newcommand{\dom}{\divgenom{}{}}

\newcommand{\dcom}{\divgenom{}{\circ}}
\newcommand{\dzom}{\divgenom{0}{}}
\newcommand{\dzg}{\divgeng{0}{}}
\newcommand{\dz}{\divgen{0}{}}

\newcommand{\dcnhat}{\divgen{\hat\Gamma_{\nu,k}}{\circ}}

\newcommand{\dcn}{\divgen{\Gamma_\nu}{\circ}}

\newcommand{\dcng}{\divgeng{\gamma_\nu}{\circ}}
\newcommand{\dctg}{\divgeng{\Gamma_\tau}{\circ}}
\newcommand{\cdcnhat}{\cdivgen{\hat\Gamma_{\nu,k}}{\circ}}
\newcommand{\dcz}{\divgen{0}{\circ}}
\newcommand{\dczng}{\divgeng{\gamma_\nu,0}{\circ}}

\newcommand{\dcnom}{\divgenom{\Gamma_\nu}{\circ}}

\newcommand{\dcznom}{\divgenom{\Gamma_\nu,0}{\circ}}
\newcommand{\dcztom}{\divgenom{\Gamma_\tau,0}{\circ}}

\DeclareMathOperator{\cdivspace}{\mathcal{D}}
\newcommand{\cdivgen}[2]{\overset{#2}{\cdivspace}{}_{#1}}
\newcommand{\cdivgenom}[2]{\cdivgen{#1}{#2}(\om)}
\newcommand{\cdivgeng}[2]{\cdivgen{#1}{#2}(\Xi)}

\newcommand{\cdcnom}{\cdivgenom{\Gamma_\nu}{\circ}}
\newcommand{\cdcznom}{\cdivgenom{\Gamma_\nu,0}{\circ}}
\newcommand{\cdcztom}{\cdivgenom{\Gamma_\tau,0}{\circ}}

\newcommand{\cdcom}{\cdivgenom{}{\circ}}

\newcommand{\cdcng}{\cdivgeng{\gamma_\nu}{\circ}}

\newcommand{\cdczng}{\cdivgeng{\gamma_\nu,0}{\circ}}
\newcommand{\cdcztg}{\cdivgeng{\gamma_\tau,0}{\circ}}

\newcommand{\hoom}{\Hoom}
\newcommand{\hocom}{\Hocom}

\newcommand{\non}{\nonumber}

\renewcommand{\norm}[1]{\left|#1\right|}

\newcommand{\normb}[1]{\big|#1\big|}

\def\undertilde#1{\mathord{\vtop{\ialign{##\crcr
$\hfil\displaystyle{#1}\hfil$\crcr\noalign{\kern1.5pt\nointerlineskip}
$\hfil\tilde{}\hfil$\crcr\noalign{\kern1.5pt}}}}}
\def\wideundertilde#1#2{\mathord{\vtop{\ialign{##\crcr
$\hfil\displaystyle{#2}\hfil$\crcr\noalign{\kern2pt\nointerlineskip}
$\hfil\widetilde{\hspace*{#1mm}}\hfil$\crcr\noalign{\kern2pt}}}}}
\def\utilde#1{\undertilde{#1}}

\title[\sc The Maxwell Compactness Property]
{\Large\sf The Maxwell Compactness Property\\
in Bounded Weak Lipschitz Domains
with Mixed Boundary Conditions}
\author{Sebastian Bauer}
\author{Dirk Pauly}
\author{Michael Schomburg}
\address{Fakult\"at f\"ur Mathematik,
Universit\"at Duisburg-Essen, Campus Essen, Germany}
\email[Sebastian Bauer]{sebastian.bauer.seuberlich@uni-due.de}
\email[Dirk Pauly]{dirk.pauly@uni-due.de}
\email[Michael Schomburg]{michael.schomburg@uni-due.de}
\date{\today}
\thanks{This contribution has been published in \cite{bauerpaulyschomburgmaxcompweaklip}.}

\begin{document}


%


\begin{abstract}
Let $\om\subset\rt$ be a bounded weak Lipschitz domain with boundary $\Gamma:=\p\om$
divided into two weak Lipschitz submanifolds $\Gamma_\tau$ and $\Gamma_\nu$
and let $\eps$ denote an $\Li$-matrix field inducing an inner product in $\Ltom$.
The main result of this contribution is the so called `Maxwell compactness property',
that is, the Hilbert space
$$\big\{E\in\Ltom\,:\,\rot E\in\Ltom,\,\div\eps E\in\Ltom,\,
\nu\times E|_{\Gamma_{\tau}}=0,\,\nu\cdot\eps E|_{\Gamma_{\nu}}=0\big\}$$
is compactly embedded into $\Ltom$. We will also prove some canonical applications,
such as Maxwell estimates, Helmholtz decompositions and
a static solution theory. Furthermore, a Fredholm alternative 
for the underlying time-harmonic Maxwell problem and 
all corresponding and related results for exterior domains 
formulated in weighted Sobolev spaces are straightforward.
\end{abstract}


\maketitle
\tableofcontents
\newpage


\section{Introduction}

One of the main and most important tools in the theory of Maxwell's equations 
is the compact embedding of vector fields possessing weak divergence and rotation in $\Lt$, 
subject to appropriate possibly mixed boundary conditions, into $\Lt$.

Let $ \om\subset\rt $ denote a bounded domain with boundary $ \Gamma := \p\om $, 
where $ \Gamma $ is divided into two relatively open subsets $ \Gamma_\tau $ 
and its complement $ \Gamma_\nu := \Gamma\setminus\ol\Gamma_\tau $. 
Furthermore, let $ \eps :\om\rightarrow\rz^{3\times 3} $ 
denote a symmetric and uniformly positive definite $ \Li $-matrix field, 
which will throughout the paper be called admissible. 
The so-called Maxwell compactness property, i.e., the compactness of the embedding
\begin{align}
\label{MCP}
\rctom\cap\eps^{-1}\dcnom\hookrightarrow\Ltom,
\end{align}
has been investigated in various settings. Here $ \rctom\cap\eps^{-1}\dcnom $ 
denotes the space of all $ E\in\Ltom $ with $ \rot E \in \Ltom $ and $ \div \eps E \in \Ltom $ 
satisfying the mixed boundary conditions $ \nu \times E = 0 $ 
on $ \Gamma_\tau $ and $ \nu \cdot \eps E = 0 $ on $ \Gamma_\nu $ in a weak sense.

Historically and e.g. for full tangential boundary conditions, \eqref{MCP} has first been proved 
by a regularity argument, showing that in a sufficiently smooth setting
$\rcom\cap\eps^{-1}\dom$ is continuously embedded into $\Hoom$ (Gaffney's inequality) 
and hence contains a $\Ltom$-converging subsequence by Rellich's selection theorem.

A first result for non-smooth, more precisely cone-like, i.e., more or less strong Lipschitz, 
domains in $ \rz^N $ or even Riemannian manifolds, 
was obtained by Weck \cite{weckmax} in the case of full homogeneous boundary conditions, 
i.e., $ \Gamma_\nu = \emptyset $ or $ \Gamma_\tau = \emptyset $, 
using the general setting provided by the calculus of alternating differential forms. 
Weber \cite{webercompmax} 
found a new proof for bounded domains in $ \rt $ satisfying the uniform cone-condition, 
which is again more or less strong Lipschitz. 
This result was improved upon by Witsch \cite{witschremmax}, 
who showed that the compact embedding is valid for bounded domains of $ \rt $ 
satisfying merely the p-cusp condition for $ 1 < p < 2 $. An elementary proof for weak Lipschitz domains, 
which even holds for weak Lipschitz manifolds, was given by Picard \cite{picardcomimb}. 
Costabel \cite{costabelremmaxlip} proved the compact 
embedding by means of a weak regularity result, i.e., there holds the continuous embedding 
into $\Sobolev^{\nicefrac{1}{2}}$ and therefore the compact embedding 
into $\Sobolev^t$ for $ t < 1/2$ and into  $\Sobolev^0 = \Lt$ in particular. 
All these results have been obtained for full boundary conditions.

Kuhn \cite{kuhndiss}, using the methods developed by Weck \cite{weckmax} and comparable assumptions 
about the regularity of the boundary and the interface, obtained the compact embedding for mixed boundary conditions 
in the setting of differential forms.
Based on the techniques developed by Weber \cite{webercompmax}, Jochmann \cite{jochmanncompembmaxmixbc} 
showed the compact embedding to hold for vector fields satisfying mixed boundary conditions, if $ \om $ 
has a (strong) Lipschitz continuous boundary $ \Gamma $ 
with a (strong) Lipschitz continuous interface $ \ol\Gamma_\tau \cap \ol\Gamma_\nu $. 
More precisely, the boundary $ \Gamma $ and the interface 
can be locally represented as graphs of Lipschitz functions.

In this paper it is shown that the assumptions in \cite{jochmanncompembmaxmixbc} 
can be weakened to include domains $ \om $ with weak Lipschitz boundaries and weak Lipschitz interfaces, 
i.e., the boundary $ \Gamma $ is assumed to be a Lipschitz manifold 
and the interface a Lipschitz submanifold of $ \Gamma $. 
This assumption is weaker than the assumptions in \cite{jochmanncompembmaxmixbc}. 
Moreover, weak Lipschitz domains are relevant in various applications. 
A prominent example is the so-called two brick domain.

Our paper closely follows \cite{jochmanncompembmaxmixbc} and hence \cite{webercompmax}. 
Conveniently, the proofs from \cite{jochmanncompembmaxmixbc} carry over practically verbatim. 
However, a modification of \cite[Theorem 1]{jochmanncompembmaxmixbc} (see our Theorem \ref{satzD0L6}) 
allows for a shorter and more straightforward proof of the compact embedding.
Moreover, we will use different transformations (charts), which reduces the compact embedding
to the flat situation of a half-cube.

The main result of this contribution is the compact embedding stated in Theorem \ref{satzMKE}. 
In the last chapter we present applications of the main theorem, i.e. the Maxwell estimate, 
Helmholtz decompositions and, following the approach developed 
by Picard in \cite{picardpotential} and \cite{picardboundaryelectro}, 
a solution theory for the static Maxwell problem involving mixed boundary conditions.

Another application is the proof of Fredholm's alternative for the time-harmonic Maxwell's equations.
This is straightforward but would be beyond the scope of this paper.

Other important applications can be mentioned in connection 
with the treatment of static or time-harmonic Maxwell's equations
in exterior domains $\om_{\textsf{ext}}\subset\rt$, i.e., domains with compact complement.
In this case the compact embedding \eqref{MCP} no longer holds. On the other hand \eqref{MCP} immediately implies
the so-called local Maxwell compactness property, i.e., the compactness of the embedding
\begin{align}
\label{LMCP}
\rct(\om_{\textsf{ext}})\cap\eps^{-1}\dcn(\om_{\textsf{ext}}) 
\hookrightarrow \Lebesgue^2_{\textsf{loc}}(\om_{\textsf{ext}}).
\end{align}
This compact embedding is the main and 
and most important tool for showing, e.g., 
polynomially weighted Maxwell inequalities and Helmholtz decompositions
or the validity of Eidus' principles of limiting absorption 
and limiting amplitude \cite{eiduslabp,eiduslamp,eiduslamptwo},
which are fundamental in the theory of static or time-harmonic Maxwell's equations in exterior domains, see e.g.
\cite{leisbook,
picardpotential,picardboundaryelectro,picardlowfreqmax,picarddeco,picardweckwitschxmas,
kuhnpaulyregmax,
paulytimeharm,paulystatic,paulydeco,paulyasym,paulypoly}.

\section{Notation, preliminaries and outline of the proof}

Let $ \om\subset\rt $ be a domain, i.e., an open and connected set. We introduce the function spaces
\begin{align*}
\Cont^{0,1}(\om)&:=\left\{u:\om\rightarrow\rz:u\text{ Lipschitz continuous}\right\},\\
\Cont^k(\om)&:=\left\{u:\om\rightarrow\rz:u\text{ is k-times continuously partially differentiable}\right\},\\
\Cont^{\infty}(\om)&:=\bigcap_{k\in\nz} \Cont^k(\om),\\
\Czoc(\om)&:=\left\{u\in\Cont^{0,1}(\om): \supp u\Subset\om\right\},\\
\Cic(\om)&:=\left\{u\in\Cont^{\infty}(\om): \supp u\Subset \om\right\},
\end{align*}
where $\Theta\Subset \om$ means $\ol{\Theta}$ is compact and a subset of $\om$.
The usual Lebesgue and Sobolev spaces will be denoted by
$$\Ltom,\quad\Ltperpom:=\Ltom\cap\rz^{\bot},\qquad\Hoom,\quad\Hoperpom:=\Hoom\cap\rz^{\bot},$$
where $\bot$ means orthogonality in $\Ltom$.
We also introduce the Sobolev (Hilbert) spaces
\begin{align*}
\rom:=\left\{E\in\Ltom: \rot E \in\Ltom\right\},\quad \dom:=\left\{E\in\Ltom: \div E \in\Ltom\right\}
\end{align*}
in the distributional sense and define the test functions or vector fields
\begin{align*}
\Cictom := \{ \varphi|_\om : \varphi\in\Cic(\rt),~ \dist(\supp \varphi, \Gamma_\tau) > 0\}
\end{align*}
and 
\begin{align*}
\Czoct(\om) := \{ \varphi|_\om : \varphi\in\Czoc(\rt),~ \dist(\supp \varphi, \Gamma_\tau) > 0\}.
\end{align*}
Note that $ \Cic_\emptyset = \Ci(\ol\om) $ and $ \Czoc_\emptyset = \Czo(\ol\om) $.
Now define 
\begin{align}
\label{defstark}
\Hoctom :=\overline{\Cictom}^{\hoom}, \quad \rctom :=\overline{\Cictom}^{\rom}, \quad
\dcnom :=\overline{\Cicnom}^{\dom}
\end{align}
as closures of test functions respectively fields.
For $ \Gamma_\tau = \Gamma $ (resp. $ \Gamma_\nu = \Gamma $) we set
$$\Hocom := \Hoctom,\quad\rcom := \rctom,\quad \dcom := \dcnom.$$
Moreover, we define the closed subspaces
\begin{align*}
\rzom:=\big\{E\in\rom: \rot E = 0\big\},\quad \dzom:=\big\{E\in\dom: \div E = 0\big\}
\end{align*}
as well as $ \rcztom := \rctom \cap \rzom $ and $\dcznom := \dcnom \cap \dzom $.
Furthermore, we introduce the weak spaces
\begin{align}
\label{defschwach}
\begin{split}
\cHoctom &:= \big\{u\in\Hoom:\scpLtom{u}{\div\Phi}=-\scpLtom{\na u}{\Phi}~\text{for all}~\Phi\in\Cicnom\big\},\\
\crctom &:=\big\{E\in\rom:\scpLtom{E}{\rot \Phi} = \scpLtom{\rot E}{\Phi}~\text{for all}~ \Phi\in\Cicnom\big\},\\
\cdcnom &:=\big\{H\in\dom:\scpLtom{H}{\na\varphi} = -\scpLtom{\div H}{\varphi}~\text{for all}~ \varphi\in\Cictom\big\},
\end{split}
\end{align}
and again for $ \Gamma_\tau = \Gamma $ (resp. $ \Gamma_\nu = \Gamma $) we set
$$\cHocom := \cHoctom,\quad\crcom := \crctom,\quad \cdcom := \cdcnom.$$
In \eqref{defstark} and \eqref{defschwach} homogeneous scalar, tangential 
and normal traces on $ \Gamma_\tau $, respectively $ \Gamma_\nu $, are generalized.

\begin{rem} \label{remlipclosure}\mbox{}
\begin{itemize}
\item[\bf(i)] In definitions \eqref{defstark} and \eqref{defschwach} $ \Cict(\om) $ and $ \Cicn(\om) $ can be replaced
(by mollification) by $ \Czoct(\om) $ and $ \Czocn(\om) $, respectively. 
\item[\bf(ii)] In  \eqref{defschwach}  $ \Cict(\om) $ and $ \Cicn(\om) $ can be replaced (by continuity) 
by $ \dcn(\om) $, $ \rcn(\om) $ and $ \Hoct(\om) $, respectively. 
In the special case of no boundary conditions,
for this continuity argument to hold the density of $\Ci(\ol\om)$ resp. $\Czo(\ol\om)$
in $ \dom $, $ \rom $ and $ \Hoom $, respectively, is needed. 
This is valid e.g. if $\om$ has the segment property, 
which is a rather weak assumption and basically means that $\Gamma$ is continuous.
\end{itemize}
\end{rem}

Moreover we set 
\begin{align*}
\crcztom := \crctom \cap \rzom ,\quad \cdcznom := \cdcnom \cap \dzom.
\end{align*}
Note that by switching $ \Gamma_\tau $ and $ \Gamma_\nu $ 
we can define the respective boundary conditions on the other part of the boundary as well.

\begin{lem}
\label{inklusionen}
\mbox{}
The following inclusions hold:
\begin{itemize}
\item[\bf(i)] $\Hoctom\subset\cHoctom,\quad\rctom\subset\crctom,\quad\dcnom\subset\cdcnom$
\item[\bf(ii)] $\na\Hoctom\subset\rcztom,\quad\na\cHoctom\subset\crcztom$
\item[\bf(iii)] $\rot\rctom\subset\dcztom,\quad\rot\crctom\subset\cdcztom$
\end{itemize}
\end{lem}

Later we will show that in fact for all these spaces
the strong and weak definitions of the boundary conditions coincide, i.e.,
\begin{align}
\label{introws}
\Hoctom = \cHoctom ,\quad \rctom = \crctom,\quad \dcnom = \cdcnom,
\end{align} 
which is an important feature of these Sobolev spaces.
In case of full boundary conditions this can be seen from the following perspective:
Define the unbounded linear rotation operator
\begin{align*}
\Abb{\Rot}{\Cic(\om)\subset\Ltom}{\Ltom}{E}{\rot E}.
\end{align*}
By its closure
\begin{align*}
\Abb{\overline\Rot}{\rcom\subset\Ltom}{\Ltom}{E}{\rot E}
\end{align*}
the differential operator 'rot' is extended to elements of $ \rcom $
and its adjoint
\begin{align*}
\Abb{\Rot^*}{\rom\subset\Ltom}{\Ltom}{H}{\rot H}
\end{align*}
further generalizes the operator to the larger space $ \rom $.
We have $\Rot\subset\ol{\Rot}\subset\Rot^{*}$.
Moreover
\begin{align*}
\Abb{(\Rot^*)^*}{\crcom\subset\Ltom}{\Ltom}{E}{\rot E}
\end{align*}
and since $ \overline\Rot = (\Rot^*)^* $ we in particular have
$$\crcom = D((\Rot^*)^*) = D(\overline\Rot) = \rcom,$$
without any assumptions about the regularity (or boundedness) of $\om$ or $ \p \om$.
Analogously we have $ \Grad $, $ \overline{\Grad} $, $ \Grad^* $ and $ \overline\Grad = (\Grad^*)^* $ 
for the gradient operator '$\na $' 
as well as $ \Div $, $ \overline{\Div} $, $ \Div^* $ and $ \overline\Div = (\Div^*)^* $ 
for the divergence operator 'div'. This way we also get
$$\Hocom = \cHocom,\quad\dcom = \cdcom.$$
In case of mixed boundary conditions we may consider the operator
\begin{align*}
\Abb{\Rot}{\Cictom\subset\Ltom}{\Ltom}{E}{\rot E},
\end{align*}
its closure
\begin{align*}
\Abb{\overline{\Rot}}{\rctom\subset\Ltom}{\Ltom}{E}{\rot E},
\end{align*}
and its adjoint
\begin{align*}
\Abb{\Rot^*}{\crcnom\subset\Ltom}{\Ltom}{H}{\rot H}.
\end{align*}
Then the double adjoint is again the rotation $\rot$.
Its domain of definition $D((\Rot^{*})^{*})$ is the space of all $E\in\rom$, such that
$$\forall\,\Phi\in\crcnom\qquad\scpLtom{E}{\rot\Phi}=\scpLtom{\rot E}{\Phi},$$
for which it is not clear if it coincides with $\crctom$, 
since the test fields are a priori allowed to be taken
not only from $\rcnom$ but also from the possibly larger space $\crcnom$.
The same problem occurs for the operators '$ \nabla $' and '$\div$'. 
Therefore, the equalities \eqref{introws} are not obvious consequences
of simple functional analysis but need a detailed, technical argument.

\subsection{Lipschitz domains}

Let $\om\subset\rt$ be a bounded domain with boundary $\Gamma:=\p\!\om$.
We introduce the setting we will be working in. Define (cf. Figure \ref{fig:cube})
\begin{align*}
B&:=(-1,1)^3\subset\rt, &B_{\pm}&:=\{x\in B:\pm x_3>0\}, &
B_0&:=\{x\in B:x_3 = 0\}, \\
&&B_{0,\pm}&:=\{x\in B_0:\pm x_1 > 0 \},&
B_{0,0}&:=\{ x\in B_0 : x_1 = 0 \}.
\end{align*}

\begin{defi}
\label{defilipmani}
$ \om $ is called weak Lipschitz, if the boundary $ \Gamma $ is a Lipschitz submanifold, i.e., 
if there is a finite open covering $ U_1,\dots, U_K\subset \rt $ of $ \Gamma $ 
and vector fields $ \phi_k : U_k \rightarrow B$, such that for $ k = 1,\dots, K $ 
\begin{itemize}
\item[\bf(i)] $ \phi_k \in \Czo(U_k,B) $ is bijective
and $ \psi_k := \phi_k^{-1} \in \Czo(B,U_k) $,
\item[\bf(ii)] $ \phi_k(U_k \cap \om) = B_{-}$
\end{itemize}
hold.
\end{defi}

\begin{rem}
For $ k = 1,\dots, K $ we have
$\phi_k (U_k \setminus \ol{\om}) = B_+ $ and $\phi_k (U_k \cap \Gamma) = B_0$.
\end{rem}

\begin{defi}
\label{defilipsubmani}
Let $ \om $ be weak Lipschitz. A relatively open subset
$ \Gamma_\tau $ of $ \Gamma $ is called   weak Lipschitz, 
if $ \Gamma_\tau $ is a Lipschitz submanifold of $ \Gamma $, i.e., 
there is an open covering $ U_1,\dots, U_k \subset \rt $ of $ \Gamma $ 
and vector fields $ \phi_k := U_k \rightarrow B $, 
such that for $ k = 1,\dots, K $ and in addition to (i), (ii) 
in Definition \ref{defilipmani} one of 
\begin{itemize}
\item[\bf(iii)] 
$ U_k \cap \Gamma_\tau = \emptyset$,
\item[\bf(iii$'$)] 
$ U_k \cap \Gamma_\tau = U_k \cap \Gamma \qimpl \phi_k(U_k \cap \Gamma_\tau) = B_{0} $,
\item[\bf(iii$''$)] 
$ \emptyset \neq U_k \cap \Gamma_\tau \neq U_k \cap \Gamma \qimpl \phi_k(U_k \cap \Gamma_\tau) = B_{0,-} $
\end{itemize}
holds. We define $ \Gamma_\nu := \Gamma \setminus \ol\Gamma_\tau $ 
to be the relatively open complement of $\Gamma_\tau$.
\end{defi}

\begin{defi}
A pair $ (\om,\Gamma_\tau) $ conforming to Definitions \ref{defilipmani} and
\ref{defilipsubmani} will be called weak Lipschitz.
\end{defi}

\begin{rem}
For the cases (iii), (iii$'$) and (iii$''$)in Definition \ref{defilipsubmani} we further have
\begin{itemize}
\item[\bf(iii)] 
$ U_k \cap \Gamma_\tau = \emptyset \;\impl\; U_k \cap \Gamma_\nu = U_k \cap \Gamma \;\impl\; 
\phi_k(U_k \cap \Gamma_\nu) = B_{0}$,
\item[\bf(iii$'$)] 
$ U_k \cap \Gamma_\tau = U_k \cap \Gamma \;\impl\; U_k \cap \Gamma_\nu = \emptyset $,
\item[\bf(iii$''$)] 
$ \emptyset \neq U_k \cap \Gamma_\tau \neq U_k \cap \Gamma \;\impl \;
\emptyset \neq U_k \cap \Gamma_\nu \neq U_k \cap \Gamma \;\impl\; \phi_k(U_k \cap \Gamma_\nu) = B_{0,+} $
and $ \phi_k(U_k \cap \ol{\Gamma}_\tau \cap \ol{\Gamma}_\nu) = B_{0,0} $.
\end{itemize}
\end{rem}

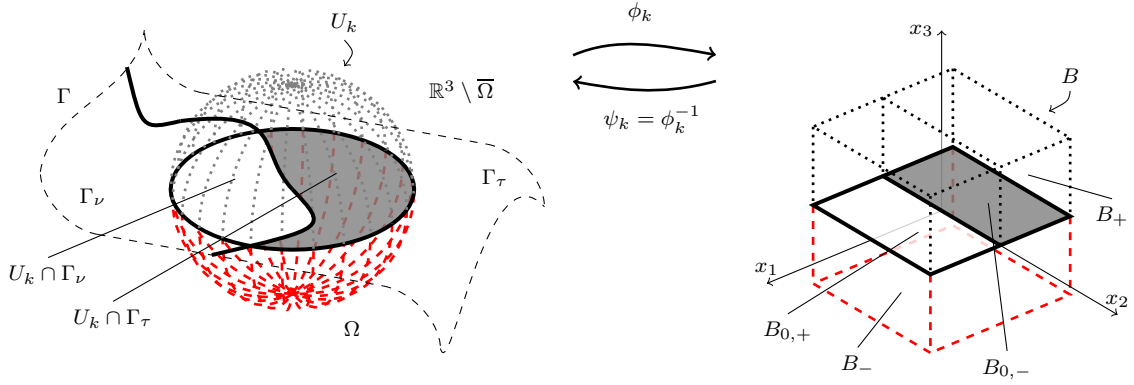
\begin{figure}
\centering
{\footnotesize		
\tdplotsetmaincoords{60}{20}
\begin{tikzpicture}[scale=0.4,tdplot_main_coords]
	\coordinate (O) at (0,0,0);		
	\def\Rad{4}				
	\foreach \angle in {0,15,...,165}{
		\tdplotsetthetaplanecoords{\angle}
		\tdplotdrawarc[red,dashed,line width=1pt,tdplot_rotated_coords]{(O)}{\Rad}{90}{270}{}{}
	}
	\fill [gray,opacity=.8] (O) circle (\Rad);
    	\fill [white] (0,0,0) -- (-1,1.3,0) -- (-2.35,3.3,0) -- (-3.3,2.3,0) -- (-3.95,0.9,0) -- (-3.95,-0.5,0) -- (-3.95,-0.9,0) -- (-3.3,-2.3,0) -- (-2.35,-3.3,0) -- (-1.1,-3.9,0) -- (0,-4,0) -- (1.3,-2.5,0) -- (1.4,-1.9,0) -- (1.1,-1.1,0) -- cycle;
	\draw [line width=1.5pt] (O) circle (\Rad);
	\foreach \angle in {-90,-75,...,75}{
		\tdplotsetthetaplanecoords{\angle}
		\tdplotdrawarc[gray,dotted,line width=1pt,tdplot_rotated_coords]{(O)}{\Rad}{-90}{90}{}{}
	}
	\draw [dashed] (-5,-5,0) -- (5,-5,0);
	\draw [dashed] (-5,-5,0) .. controls (-6,-5,0) and (-7,-5,1) .. (-7,-5,2);	
	\draw [dashed] (5,-5,0) .. controls (6,-5,0) and (7,-5,-1) .. (7,-5,-2);	
	\draw [dashed] (-5,5,0) -- (5,5,0);
	\draw [dashed] (-5,5,0) .. controls (-6,5,0) and (-7,5,1) .. (-7,5,2);	
	\draw [dashed] (5,5,0) .. controls (6,5,0) and (7,5,-1) .. (7,5,-2);		
	\draw [dashed] (-7,-5,2) .. controls (-7,-4,4) and (-7,4,0) .. (-7,5,2);	
	\draw [dashed] (7,-5,-2) .. controls (7,-4,-4) and (5,5,0) .. (7,5,-2);	
	\draw [line width=1.5pt] (-6.7,2.5,2) .. controls (-6,2,0) .. (-5,3,0);
	\draw [line width=1.5pt] (-5,3,0) .. controls (-2,5,0) and (-2,2,0) .. (0,0,0);
	\draw [line width=1.5pt] (0,0,0) .. controls (2,-2,0) .. (-1,-5,0);
	\draw (5,-8,0) 		node [] {$\Omega$};							
	\draw (3,8,0) 		node [] {$\mathbb{R}^3 \setminus \ol\Omega$};
	\draw (-8,0,2) 		node [] {$\Gamma$};							
	\draw (-6,-2.7,0) 	node [] {$\Gamma_\nu$};
	\draw (6,3,0) 		node [] {$\Gamma_\tau$};
	\draw [black,->] (0,5,3) node[anchor=south]{$U_k$} .. controls (0,6,2) .. (0,5,2);
	\draw [black] (-3,-9,0) node[anchor=north]{$U_k \cap \Gamma_\tau$} -- (1,1.5,0);
	\draw [black] (-6,-7,0) node[anchor=north]{$U_k \cap \Gamma_\nu$} -- (-2,0,0);
	\draw [black,line width=1pt,->] (8,5,4) .. controls (9.5,5,5) and (11,5,5) .. (13,5,5);
	\draw (10,6,5.5) node [] {$\phi_{k}$};
	\draw [black,line width=1pt,<-] (8,5,3) .. controls (9.5,5,3) and (11,5,3) .. (13,5,4);
	\draw (10.3,6.3,1.3) node [] {$\psi_{k} = \phi_{k}^{-1}$};
\end{tikzpicture}
\tdplotsetmaincoords{60}{140}
\begin{tikzpicture}[scale=1.2,tdplot_main_coords]
	\draw [black] (0,0,0) -- (1,0,0);
	\draw [black] (0,0,0) -- (0,1,0);
	\draw [gray] (0,0,0) -- (0,0,1);
	\draw [red,dashed,line width=1pt] (1,1,-1) -- (1,-1,-1) -- (-1,-1,-1) -- (-1,1,-1) -- cycle;	
	\draw [red,dashed,line width=1pt] (1,1,0) -- (1,1,-1);							
	\draw [red,dashed,line width=1pt] (1,-1,0) -- (1,-1,-1);							
	\draw [red,dashed,line width=1pt] (-1,-1,0) -- (-1,-1,-1);						
	\draw [red,dashed,line width=1pt] (-1,1,0) -- (-1,1,-1);							
	\fill [gray,opacity=.8] (0,1,0) -- (-1,1,0) -- (-1,-1,0) -- (0,-1,0) -- cycle;				
	\fill [white,opacity=.8] (0,1,0) -- (1,1,0) -- (1,-1,0) -- (0,-1,0) -- cycle;				
	\draw [black,line width=1.5pt] (0,-1,0) -- (0,1,0);								
	\draw [black,line width=1.5pt] (1,1,0) -- (1,-1,0) -- (-1,-1,0) -- (-1,1,0) -- cycle;		
	\draw [black,dotted,line width=1pt] (0,-1,1) -- (0,-1,0);					
	\draw [black,dotted,line width=1pt] (0,1,0) -- (0,1,1);					
	\draw [black,dotted,line width=1pt] (-1,-1,1) -- (-1,-1,0);				
	\draw [black,dotted,line width=1pt] (-1,1,1) -- (-1,1,0);					
	\draw [black,dotted,line width=1pt] (1,-1,1) -- (1,-1,0);					
	\draw [black,dotted,line width=1pt] (1,1,1) -- (1,1,0);					
	\draw [black,dotted,line width=1pt] (0,-1,1) -- (1,-1,1) -- (1,1,1) -- (0,1,1);			
	\draw [black,dotted,line width=1pt] (0,1,1) -- (-1,1,1) -- (-1,-1,1) -- (0,-1,1) -- cycle;	
	\draw [black,->] (1,0,0) -- (2.5,0,0) 	node[anchor=south]{$x_1$};
	\draw [black,->] (0,1,0) -- (0,3,0) 	node[anchor=south]{$x_2$};
	\draw [black,->] (0,0,1) -- (0,0,2.3) 	node[anchor=east]{$x_3$};
	\draw [black,->] (-1,1,1.6) 	node[anchor=south]{$B$} .. controls (-1.4,0.5,1.1) .. (-1.2,0.4,1);
	\draw [black] (2.2,0,-0.5) 	node[anchor=north]{$B_{0,+}$} -- (0.5,0.2,0);
	\draw [black] (2,3.5,0.5) 	node[anchor=north]{$B_{0,-}$} -- (-0.5,0.2,0);
	\draw [black] (1.2,0,-1.3) 	node[anchor=north]{$B_{-}$} -- (1.5,1.1,0);
	\draw [black] (-1.6,1.0,0) 	node[anchor=north]{$B_{+}$} -- (-0.9,0.4,0.3);
\end{tikzpicture}
}
\caption{Mappings $\phi_{k}$ and $\psi_{k}$ between a ball $U_k$ and the cube $B$.}
\end{figure}

In the literature a bounded domain $ \om \subset \rt $ is called (strong) Lipschitz, 
if there is an open covering $ U_1,\dots, U_K \subset \rt $ and rigid body motions 
$ R_k = A_k + a_k $, $ A_k $ orthogonal, 
$ a_k \in \rt $, $ k = 1,\dots, K $, such that with $ \xi_k \in \Czo(I^2,I) $, $ k = 1,\dots, K $, 
and $ I = (-1,1) $
$$R_k(U_k \cap \om) = \{ x \in B : x_3 < \xi_k(x')  \},\quad x' = (x_1, x_2),$$
holds. Then $R_k(U_k \cap \Gamma) = \{ x \in B:x_3 = \xi_k (x') \}$.
A relatively open subset $ \Gamma_\tau \subset \Gamma $ is called (strong) Lipschitz, 
if with $ \zeta_k \in \Czo(I,I) $
$$\emptyset\neq U_k \cap \Gamma_\tau \neq U_k \cap \Gamma \qimpl R_k(U_k \cap \Gamma_\tau) 
= \{ x\in B : x_3 = \xi_k(x'),~ x_1 < \zeta_k (x_2) \}$$
holds.With this $R_k(U_k \setminus \ol{\om}) = \{ x \in B : x_3 > \xi_k (x') \}$ and
for $\emptyset \neq U_k \cap \Gamma_\tau \neq U_k \cap \Gamma$ 
\begin{align*}
R_k(U_k \cap \Gamma_\nu) &= \{ x \in B : x_3 = \xi_k (x'),~ x_1 > \zeta_k (x_2) \}, \\
R_k(U_k \cap \ol{\Gamma}_\tau \cap \ol{\Gamma}_\nu) &= \{ x \in B : x_3 = \xi_k (x'),~ x_1 = \zeta_k (x_2) \}.
\end{align*}
It holds
\begin{itemize}
\item 
$ \om $ strong Lipschitz $ \qimpl $ $ \om $ weak Lipschitz,
\item 
$ \om $ strong Lipschitz and $ \Gamma_\tau $ strong Lipschitz $ \qimpl $ $ (\om,\Gamma_\tau) $ weak Lipschitz pair,
\end{itemize}
as by setting
\begin{align*}
\varphi_k: U_k \rightarrow B~,~\varphi_k (x) := \dvec{x_1 - \zeta(x_2)}{x_2}{x_3 - \xi(x')}
\phi_k := \varphi_k \circ R_k,\quad 
\psi_k := \phi^{-1}_k
\end{align*}
we can define Lipschitz transformations as in Definitions \ref{defilipmani} and \ref{defilipsubmani}. 

For later purposes we introduce special notations for the half-cube domain
\begin{align}
\label{halfcube}
\Xi:=B_{-},\quad \gamma:=\p\Xi
\end{align} 
and its relatively open boundary parts $\gamma_{\tau}$ and $\gamma_{\nu}:=\gamma\setminus\ol{\gamma_{\tau}}$.
We will only consider the cases 
\begin{align}
\label{halfcubegn}
\gamma_{\nu}=\emptyset,\quad\gamma_{\nu}=B_{0},\quad\gamma_{\nu}=B_{0,+}
\end{align}
and we note that $\Xi$ and $\gamma_\tau$ are strong Lipschitz.

\subsection{Outline of the proof}

Let $(\om,\Gamma_{\tau})$ be a weak Lipschitz pair for a bounded domain $\om\subset\rt$.

\begin{itemize}
\item 
As a first step, we show by elementary arguments $ \Hoctom = \cHoctom $, i.e., 
for the $ \Ho $-spaces the strong and weak definitions of the boundary conditions coincide.
\item 
In the second and essential step, we construct various $ \Ho $-potentials on simple domains,
mainly for the half-cube $ \Xi$, see \eqref{halfcube}, 
with the special boundary conditions \eqref{halfcubegn}, i.e.,
$$ \crczng = \rczng = \na\Hocgng, \quad 
\cdczng = \dczng = \rot\Hocgng, \quad 
\Lt(\Xi) = \div\Hocgng. $$
\item 
In the third step it is shown that the strong and weak definitions of the boundary conditions also coincide 
for the divergence and rotation spaces on the half-cube $ \Xi$
with the special boundary conditions \eqref{halfcubegn}, i.e.,
\begin{align}
\label{pillemann}
\rcng = \crcng, \quad\dcng = \cdcng. 
\end{align}
\item 
The fourth step proves the compact embedding on the half-cube $ \Xi$
with the special boundary conditions \eqref{halfcubegn}, i.e.,
\begin{align}
\label{MCPXi}
\rctg\cap\dcng\hookrightarrow\Lt(\Xi)
\end{align}
is compact.
\item 
In the fifth step, \eqref{pillemann} is established for weak Lipschitz domains, i.e. 
$$\rctom = \crctom,\quad\dcnom = \cdcnom.$$
\item 
In the last step, we finally prove the compact embedding \eqref{MCPXi} for weak Lipschitz pairs, i.e.,
\begin{align}
\label{MCPom}
\rctom\cap\dcnom\hookrightarrow\Ltom
\end{align}
is compact.
\end{itemize}

\section{$\mathsf H^1$-potentials}

In this section $ \Ho $-potentials for irrotational or solenoidal $\Lt$-vector fields 
or $ \Lt $-functions are obtained. 
For illustrative purposes we will first give the proofs 
for the half-cube $\Xi$ with the special boundary conditions \eqref{halfcubegn}
which will also later be used as the image of the coordinate transformation 
that flattens out the boundary of a weak Lipschitz pair, 
and then show how to adjust them for more general domains. 

We start out with a density result for $\Ho$-functions, i.e.,
the strong and weak definitions of the boundary conditions coincide for $\Ho$-functions,
which is first proved for a flat boundary and then generalized 
to weak Lipschitz pairs. The proof can be found in \cite[Lemma 2, Lemma 3]{jochmanncompembmaxmixbc}. 
For the convenience of the reader we present a simplified proof, 
using our notation, in the appendix.

\begin{lem}
\label{Hoct}
Let $\om\subset\rt$ be a bounded domain and $(\om, \Gamma_{\tau})$ be a weak Lipschitz pair as well as
$$\Hocttr(\om):=\big\{u\in\Hoom:u|_{\Gamma_{\tau}}=0\big\}.$$
Then $\cHoctom=\Hocttr(\om)=\Hoctom$.
\end{lem}

\subsection{$\mathsf H^1$-potentials without boundary conditions}

The next three lemmas ensure the existence of $ \Ho $-potentials without boundary conditions. 
Suppose $\om\subset\rt$ to be a bounded domain.

\begin{lem}
\label{RzgradHo}
Let $\om$ be strong Lipschitz and simply connected.
Then there exists a continuous linear operator
$$\T_\na:\rzom\rightarrow\Ho(\rt),$$ 
such that for all $E\in\rzom$ 
$$\na(\T_\na E)=E\quad\text{in }\om.$$
Especially $\rzom=\na\Hoom=\na\Hoperpom$
and the potential depends continuously on the data.
In particular these are closed subspaces of $\Ltom$.
\end{lem}

\begin{proof}
It is classical (standard Helmholtz decomposition) that $\rzom=\na\Hoperpom$ holds.
Using Poincar\'e's inequality the potential depends continuously on the data.
By Calderon's extension theorem we can extend any potential in $\Hoperpom$ continuously to $\Ho(\rt)$.
\end{proof}

\begin{lem}
\label{div0rot}
Let $\om$ be strong Lipschitz, such that $\rt\setminus\ol\om$ is connected (i.e. $ \Gamma $ is connected). 
Then there exists a continuous linear operator\footnote{Let $X,Y,Z$ be normed spaces. 
We call an operator $T:X\to Y\cap Z$ continuous, if $T_{Y}:X\to Y$ and $T_{Z}:X\to Z$ are continuous.} 
$$\T_{\mathsf{r}}:\Ltom\rightarrow\Ho(\rt)\cap\dz(\rt),$$ 
such that for all $H\in\dzom$ 
$$\rot(\T_{\mathsf{r}} H)=H\quad\text{in }\om.$$
Especially $\dzom=\rot\Hoom=\rot\big(\Hoom\cap\dzom\big)$
and the potential depends continuously on the data.
In particular these are closed subspaces of $\Ltom$.
\end{lem}

A proof can be found in \cite[Lemma 1]{jochmanncompembmaxmixbc}. 
For the convenience of the reader we repeat the proof using our notation.

\begin{proof}
Let $\ol\om$ be a subset of $B_{\rho}(0)$, the ball with radius $\rho>0$ centered at the origin.
Define $\Theta:=B_\rho(0)\setminus\ol\om$ and 
let $\E:\Ho(\Theta)\rightarrow\Ho(B_\rho(0))$ be a continuous linear, e.g. Calderon's, extension operator.
Because $\Theta$ is connected, $\norms{\na\,\cdot\,}_{\Lt(\Theta)}$ defines a norm on $\Hoperp(\Theta)$,
which is equivalent to the $\Ho(\Theta)$-norm by Poincar\'e's inequality.
Now define for $H\in\Ltom$ the operator $\A:\Ltom\rightarrow\Hoperp(\Theta)$ by
\begin{align*}
\forall\,\psi\in \Hoperp(\Theta)\qquad\scp{\na(\A H)}{\na\psi}_{\Lt(\Theta)}=\scpLtom{H}{\na(\E\psi)},
\end{align*}
which by the Lax-Milgram lemma is well defined, linear and continuous.
Next define for $H\in\Ltom$ the linear and continuous operator $\B:\Ltom\rightarrow\Lt(\rt)\cap\Lo(\rt)$
by
\begin{align*}
\B H:=
\begin{cases}
H & \text{in }\om,\\
-\na(\A H) & \text{in }\Theta,\\
0 & \text{in }\rt\setminus \ol{B_\rho(0)}.
\end{cases}
\end{align*}
For the Fourier transform $ \mathcal{F} $
$$\mathcal{F}(\B H)\in\Li(\rt)\cap\Lt(\rt)$$
holds and thus $ \hat H \in\Lt(\rt) $ with
$$\hat H := \frac{\xi}{|\xi|^2}\times\mathcal{F}(\B H),\quad\xi(x):=\id(x):=x.$$
Then
\begin{align*}
\T_{\mathsf{r}} H:=i\mathcal{F}^{-1} \hat H \in\Ho(\rt),
\end{align*}
and since $\xi\cdot\hat H = 0$,  
$$\div \T_{\mathsf{r}} H = -\mathcal{F}^{-1}(\xi\cdot \hat H) = 0$$
follows, i.e. $\T_{\mathsf{r}} H \in\Ho(\rt)\cap\dz(\rt)$,
and $\T_{\mathsf{r}} : \Ltom \to \Ho(\rt)\cap\dz(\rt)$ is a linear and continuous operator.
Now, suppose $\div H=0$ in $\om$. Let $\psi\in\Cicrt$, $\alpha:=|\Theta|^{-1}\int_\Theta \psi$ 
and $\tilde\psi:=\psi|_\Theta-\alpha\in \Hoperp(\Theta)$, 
where $|\Theta|$ denotes the Lebesgue measure of $\Theta$. Then
\begin{align*}
\scp{\B H}{\na\psi}_{\Lt(\rt)} = \scpLtom{H}{\na\psi} - \scp{\na(\A H)}{\na\tilde\psi}_{\Lt(\Theta)}
=\scpLtom{H}{\na(\psi - \E\tilde\psi)} = \scpLtom{H}{\na\varphi},
\end{align*}
where $\varphi:=\psi-\E\tilde\psi-\alpha\in\Hoom$. 
Since $\varphi$ vanishes on $\Theta$, we have $\varphi\in\Hocom$. $\div H=0$ in $\om$ yields 
$$\forall\,\psi\in\Cicrt\qquad\scp{\B H}{\na\psi}=0,$$
i.e., $\div (\B H)=0$ on $\rt$. Thus $\xi\cdot\mathcal{F}(\B H)= 0$.
But then $ \rot(\T_{\mathsf{r}} H) = \B H $ in $ \rt $ as
$$\rot(\T_{\mathsf{r}} H) =-\mathcal{F}^{-1}(\xi\times\hat H)=\mathcal{F}^{-1}\mathcal{F}\B H$$
and in particular $ \rot(\T_{\mathsf{r}} H) = H $ in $ \om $.
\end{proof}

Using the same method, a divergence potential for an $ \Lt $ function can be obtained.

\begin{lem}
\label{lemL2divFT}
There exists a continuous linear operator 
$$\T_{\mathsf{d}}:\Ltom\rightarrow\Ho(\rt)\cap\rotspace_0(\rt),$$ 
such that for all $h\in\Ltom$ 
$$\div(\T_{\mathsf{d}} h)=h\quad\text{in }\om.$$
Especially $\Ltom=\div\Hoom=\div\big(\Hoom\cap\rzom\big)$
and the potential depends continuously on the data.
In particular these are closed subspaces of $\Ltom$.
\end{lem}

\begin{proof}
Let $\tilde h$ be the extension of $ h $ into $ \rt $ by zero. For the Fourier transform $ \mathcal{F} $
$$\mathcal{F}(\tilde h)\in\Li(\rt)\cap\Lt(\rt)$$
holds and thus $ \hat H \in \Lt(\rt) $ with 
$$\hat H := \frac{\xi}{|\xi|^2}\mathcal{F}(\tilde h).$$
Then  
\begin{align*}
\T_{\mathsf{d}} h:=-i\mathcal{F}^{-1}\hat H \in\Ho(\rt),
\end{align*}
and since $\xi\times\hat H = 0$,  
$$\rot \T_{\mathsf{d}} h = \mathcal{F}^{-1}(\xi\times \hat H) = 0$$
follows, i.e. $\T_{\mathsf{d}} h \in\Ho(\rt)\cap\rotspace_0(\rt)$,
and $\T_{\mathsf{d}}$ is a continuous linear operator from $\Ltom$ to $\Ho(\rt)\cap\rotspace_0(\rt)$. 
Finally, $ \div \T_{\mathsf{d}} h = \tilde h $ in $ \rt $ as
$$\div(\T_{\mathsf{d}} h) =\mathcal{F}^{-1}(\xi\cdot\hat H)=\mathcal{F}^{-1}\mathcal{F}\B\, \tilde h$$
and in particular $ \div \T_{\mathsf{d}} h = h $ in $ \om $.
\end{proof}

\begin{rem}
Let $\Theta\subset\rt$ be a domain with $\ol\om\subset\Theta$.
Using a cutting technique we can choose continuous linear potential operators
$\T_\na:\rzom\rightarrow\Hoc(\Theta)$, 
$\T_{\mathsf{r}}:\Ltom\rightarrow\Hoc(\Theta)$
and 
$\T_{\mathsf{d}}:\Ltom\rightarrow\Hoc(\Theta)$ 
with $\na(\T_\na E)=E$, 
$\rot(\T_{\mathsf{r}} H)=H$ 
and 
$\div(\T_{\mathsf{d}} h)=h$
for $E\in\rzom$,
$H\in\dzom$
and
$h\in\Ltom$, respectively.
\end{rem}

\subsection{$\mathsf H^1$-potentials with boundary conditions}
\label{secpotwithbc}

Now we start constructing $ \Ho $-potentials with boundary conditions. 
Let us recall our special setting on the half-cube
$$\Xi=B_{-}\quad\text{and}\quad
\gamma_{\nu}=\emptyset,\quad\gamma_{\nu}=B_{0}\quad\text{or}\quad\gamma_{\nu}=B_{0,+}.$$
Furthermore, cf. Figure \ref{fig:cube}, we extend $ \Xi $ over $ \gamma_\nu $ by
\begin{align*}
\tilde \Xi&\,=
\textrm{int}(\ol\Xi\cup\ol{\hat \Xi}),&
\hat \Xi&:=
\begin{cases}
\{x\in B:x_3>0\}=B_{+}
&\text{, if }\gamma_{\nu}=B_{0},\\
\{x\in B:x_3>0,\,x_1>0\}
=\{x\in B_{+}:x_1>0\}=:B_{+,+}
&\text{, if }\gamma_{\nu}=B_{0,+}.
\end{cases}
\end{align*}

\begin{figure}
\centering
{\footnotesize		
\tdplotsetmaincoords{60}{140}
\begin{tikzpicture}[scale=1.6,tdplot_main_coords]
	\draw [black] (0,0,0) -- (1,0,0);
	\draw [black] (0,0,0) -- (0,1,0);
	\draw [gray] (0,0,0) -- (0,0,1.4);
	\draw [red,dashed,line width=1pt] (1,1,-1) -- (1,-1,-1) -- (-1,-1,-1) -- (-1,1,-1) -- cycle;	
	\draw [red,dashed,line width=1pt] (1,1,0) -- (1,1,-1);							
	\draw [red,dashed,line width=1pt] (1,-1,0) -- (1,-1,-1);							
	\draw [red,dashed,line width=1pt] (-1,-1,0) -- (-1,-1,-1);						
	\draw [red,dashed,line width=1pt] (-1,1,0) -- (-1,1,-1);							
	\fill [white,opacity=.7] (1,1,0) -- (1,-1,0) -- (-1,-1,0) -- (-1,1,0) -- cycle;				
	\draw [black,line width=1.5pt] (1,1,0) -- (1,-1,0) -- (-1,-1,0) -- (-1,1,0) -- cycle;		
	\draw [black,dotted,line width=1pt] (1,-1,1) -- (1,-1,0);							
	\draw [black,dotted,line width=1pt] (1,1,1) -- (1,1,0);							
	\draw [black,dotted,line width=1pt] (-1,-1,1) -- (-1,-1,0);						
	\draw [black,dotted,line width=1pt] (-1,1,1) -- (-1,1,0);							
	\fill [white,opacity=.5] (-1,-1,1) -- (1,-1,1) -- (1,1,1) -- (-1,1,1) -- cycle;				
	\draw [black,dotted,line width=1pt] (-1,-1,1) -- (1,-1,1) -- (1,1,1) -- (-1,1,1) -- cycle;	
	\draw [black,->] (1,0,0) -- (2.5,0,0) 	node[anchor=south]{$x_1$};
	\draw [black,->] (0,1,0) -- (0,3,0) 	node[anchor=south]{$x_2$};
	\draw [black,->] (0,0,1) -- (0,0,2.3) 	node[anchor=east]{$x_3$};
  \draw [color=black,->] (3,1,0) -- (1.5,1,0);
  \draw [color=black,->] (3.1,1,2) -- (1.5,1,1);
  \draw [color=black,->] (0.7,2,3) -- (1.2,1,2);
	\draw (0.7,2.4,3.2) node [] {\Large$\hat{\Xi}$};
  \draw (3.4,1.1,0.2) node [] {\Large$\Xi$};
  \draw (3.4,1,2.1) node [] {\Large$B_{0}$};
\end{tikzpicture}
\hspace*{5mm}
\tdplotsetmaincoords{60}{140}
\begin{tikzpicture}[scale=1.6,tdplot_main_coords]
	\draw [black] (0,0,0) -- (1,0,0);
	\draw [black] (0,0,0) -- (0,1,0);
	\draw [black] (0,0,0) -- (0,0,1);
	\draw [red,dashed,line width=1pt] (1,1,-1) -- (1,-1,-1) -- (-1,-1,-1) -- (-1,1,-1) -- cycle;	
	\draw [red,dashed,line width=1pt] (1,1,0) -- (1,1,-1);							
	\draw [red,dashed,line width=1pt] (1,-1,0) -- (1,-1,-1);							
	\draw [red,dashed,line width=1pt] (-1,-1,0) -- (-1,-1,-1);						
	\draw [red,dashed,line width=1pt] (-1,1,0) -- (-1,1,-1);							
	\fill [white,opacity=.8] (-1,-1,0) -- (1,-1,0) -- (1,1,0) -- (-1,1,0) -- cycle;				
	\draw [red,dashed,line width=1pt] (0,-1,0) -- (0,1,0) -- (-1,1,0) -- (-1,-1,0) -- cycle; 	
	\draw [black,line width=1.5pt] (1,1,0) -- (1,-1,0) -- (0,-1,0) -- (0,1,0) -- cycle;		
	\draw [black,dotted,line width=1pt] (0,-1,1) -- (0,-1,0);							
	\draw [black,dotted,line width=1pt] (1,-1,1) -- (1,-1,0);							
	\draw [black,dotted,line width=1pt] (0,1,0) -- (0,1,1);							
	\draw [black,dotted,line width=1pt] (1,1,1) -- (1,1,0);							
	\fill [white,opacity=.8] (0,-1,1) -- (1,-1,1) -- (1,1,1) -- (0,1,1) -- cycle;				
	\draw [black,dotted,line width=1pt] (0,-1,1) -- (1,-1,1) -- (1,1,1) -- (0,1,1) -- cycle;	
	\draw [black,->] (1,0,0) -- (2.5,0,0) node[anchor=south]{$x_1$};
	\draw [black,->] (0,1,0) -- (0,3,0) node[anchor=south]{$x_2$};
	\draw [black,->] (0,0,1) -- (0,0,2) node[anchor=east]{$x_3$};
  \draw [color=black,->] (3,1,0) -- (1.5,1,0);
  \draw [color=black,->] (1,2,2.5) -- (1,1,1.4);
  \draw [color=black,->] (3.1,1,2) -- (1.5,1,1);
	\draw (1,2.4,2.8) node [] {\Large$\hat{\Xi}$};
  \draw (3.4,1.1,0.2) node [] {\Large$\Xi$};
  \draw (3.5,1,2.1) node [] {\Large$B_{0,+}$};
\end{tikzpicture}
}
\caption{The half-cube $\Xi=B_{-}$, extended by $\hat\Xi$ to the polygonal domain $\tilde\Xi$,
and the rectangles $\gamma_{\nu}=B_{0}$ and $\gamma_{\nu}=B_{0,+}$.} 
\label{fig:cube}
\end{figure}
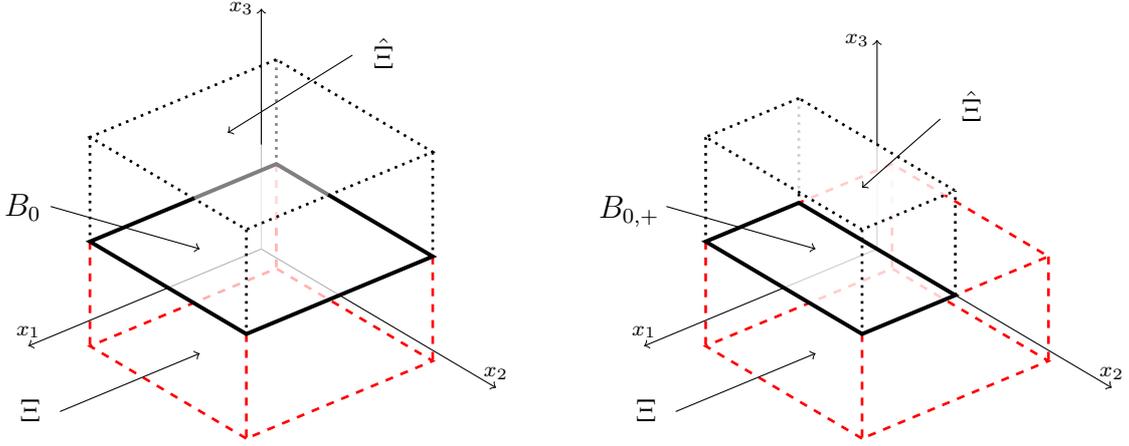

\begin{theo}
\label{gradphirot0}
There exists a continuous linear operator
$$\mathcal{S}_{\na}:\crczng\rightarrow\Ho(\rt)\cap\Hocgn(\Xi),$$
such that for all $E\in\crczng$
$$\na(\mathcal{S}_{\na}E)=E\quad\text{in }\Xi.$$
Especially $\crczng=\rczng=\na\Hocgn(\Xi)$
and the potential depends continuously on the data.
In particular these are closed subspaces of $\Ltom$.
\end{theo}

\begin{rem}
\label{gradphirot0rem}
The latter theorem is also from \cite[Lemma 4]{jochmanncompembmaxmixbc}. 
Nevertheless, we will give a modified and simplified proof. 
Moreover, as the proof will show the result holds for more general domains. 
Let $ \om\subset\rt $ be a bounded and simply connected domain and
let $(\om,\Gamma_{\nu})$ be a weak Lipschitz pair, such that $\Gamma_{\nu}$ is connected. Then a potential operator
$\mathcal{S}_{\na}:\crcznom\rightarrow\Hocn(\om)$ exists and 
$$\crcznom=\rcznom=\na\Hocn(\om).$$
If $\om$ is even strong Lipschitz, a continuous potential operator
$\mathcal{S}_{\na}:\crcznom\rightarrow\Ho(\rt)\cap\Hocn(\om)$ can be chosen.
\end{rem}

\begin{proof}
The case $\gamma_{\nu}=\emptyset$ is well known, i.e., $\crotgeng{0}{} = \rzg = \na\Ho(\Xi)=\na\Hoperp(\Xi)$, see Lemma \ref{RzgradHo}.
The other two cases can be treated together. We have
\begin{align*}
\na\Hocgn(\Xi)\subset\rczng\subset\crczng
\end{align*}
and need to show $\crczng\subset\na\Hocgn(\Xi)$.
Suppose $E\in\crczng$ and define $\tilde E\in\Lt(\tilde \Xi)$ by
\begin{align*}
\tilde E:=
\begin{cases}
E & \text{in }\Xi, \\
0 & \text{in }\hat \Xi.
\end{cases}
\end{align*}
It follows $\rot \tilde E=0$ in $\tilde \Xi$, as for any $\Phi\in\Cic(\tilde \Xi)$, 
due to $\ol\gamma_{\tau}\subset\p \tilde \Xi$, also $\Phi\in\Cictg$ and thus
$$0=\scpLtXi{E}{\rot\Phi} = \scp{\tilde E}{\rot\Phi}_{\Lt(\tilde\Xi)},$$
which means $\tilde E\in\rotspace_0(\tilde \Xi)$.
Because $\tilde \Xi$ is simply connected, there exists a potential $u\in\Ho(\tilde \Xi)$ with
\begin{align*}
\tilde E = \na u ~\text{in}~\tilde \Xi.
\end{align*}
In particular $\na u = 0$ in $\hat \Xi$ which implies $ u = c $ in $ \hat \Xi $ for some constant $ c\in\rz $. Define
\begin{align*}
\tilde u := u - c \in\Ho(\tilde \Xi).
\end{align*}
Note that $\tilde u = 0$ in $\hat \Xi$, so $\tilde u|_{\gamma_{\nu}} = 0$ follows, which means $\tilde u\in\Hocgntr(\Xi)$ 
and with Lemma \ref{Hoct} ($\om=\Xi$, $\Gamma_\tau=\gamma_\nu$) we conclude $\tilde u \in\Hocn(\Xi)$. Moreover
$\na\tilde u=\na u=E $ in $\Xi$. By Poincar\'e's inequality the potential $\tilde u$ depends continuously on the data $E$.
Using Calderon's extension theorem we can extend any potential in $\Hocgn(\Xi)$ continuously to $\Ho(\rt)$.
\end{proof}

Next up is the existence of an $ \Ho $-potential for divergence free fields subject to the special normal boundary condition. 
This Theorem is a modification of \cite[Theorem 1]{jochmanncompembmaxmixbc}, where the potential is only in $ \crcn $. 

\begin{theo}
\label{satzD0L6}
There exists a continuous linear operator
$$\mathcal{S}_{\mathsf{r}}:\cdczng\rightarrow\Ho(\rt)\cap\Hocgng,$$
such that for all $H\in\cdczng$
$$\rot(\mathcal{S}_{\mathsf{r}} H)=H\quad\text{in }\Xi.$$
Especially $\cdczng=\dczng=\rot\Hocgn(\Xi)=\rot\rcng=\rot\crcng$
and the $\Hocgn(\Xi)$-potential depends continuously on the data.
In particular these are closed subspaces of $\Ltom$.
\end{theo}

\begin{proof}
We start with the case $\gamma_{\nu}=\emptyset$.
By Lemma \ref{div0rot} $\cdivgeng{0}{} = \dzg = \rot\Ho(\Xi)$ 
and we can set $\mathcal{S}_{\mathsf{r}} H:=\T_{\mathsf{r}} H\in\Ho(\rt)$
with $\rot(\mathcal{S}_{\mathsf{r}} H)=H$ in $\Xi$. 
The other two cases can be treated together.
Suppose $H\in\cdczng$ and define $\tilde H\in\Lt(\tilde \Xi)$ by
\begin{align}
\label{defHschlange}
\tilde H:=
\begin{cases}
H&\text{in }\Xi,\\
0&\text{in }\hat \Xi.
\end{cases}
\end{align}
It follows $\div \tilde H=0$ in $\tilde \Xi$, as for any $\psi\in\Cic(\tilde \Xi)$, 
due to $\ol\gamma_{\tau}\subset\p \tilde \Xi$, also $\psi\in\Cictg$ and thus
$$0= \scpLtXi{H}{\na\psi} = \scpLtXit{\tilde H}{\na\psi},$$
which means $\tilde H\in\dz(\tilde \Xi)$. Because $\rt\setminus\ol{\tilde \Xi}$ is connected, Lemma \ref{div0rot} yields
$\T_{\mathsf{r}} \tilde H\in\Ho(\rt)\cap\dz(\rt)$
with $\rot(\T_{\mathsf{r}} \tilde H)=\tilde H$ in $\tilde \Xi$.
In particular $\T_{\mathsf{r}} \tilde H\in\Ho(\hat \Xi)$ and $\rot(\T_{\mathsf{r}} \tilde H) = 0$ in $\hat \Xi$.
Because $\hat \Xi$ is simply connected, there exists a unique $\varphi\in\Hoperp(\hat \Xi)$ with
\begin{align*}
\T_{\mathsf{r}} \tilde H=\na\varphi\quad\text{in }\hat \Xi.
\end{align*}
Since $ \T_{\mathsf{r}} \tilde H \in \Ho(\hat \Xi) $ we have $\varphi\in \Ht(\hat \Xi)$.
Let $\E:\Ht(\hat \Xi)\rightarrow \Ht(\rt)$ be a continuous, linear extension operator, for example Calderon's. Then
\begin{align*}
\Abb{\mathcal{S}_{\mathsf{r}}}{\cdczng}{\Ho(\rt)}{H}{\T_{\mathsf{r}} \tilde H-\na(\E\varphi)}
\end{align*}
is linear and continuous.
Since $\mathcal{S}_{\mathsf{r}} H = 0$ in $\hat \Xi$, we have $\mathcal{S}_{\mathsf{r}} H|_{\gamma_{\nu}} = 0$, 
which means $\mathcal{S}_{\mathsf{r}} H\in\Hocgntr(\Xi)$.
Hence $\mathcal{S}_{\mathsf{r}} H \in\Hocgn(\Xi)\subset\rcng\subset\crcng$ by Lemma \ref{Hoct}. Moreover
$$\rot(\mathcal{S}_{\mathsf{r}} H)= H\quad\text{in }\Xi,$$
as $\rot(\mathcal{S}_{\mathsf{r}} H) = \rot(\mathcal{T}_{\mathsf{r}} \tilde H) = \tilde H$ even in $\tilde \Xi$.
Recalling Lemma \ref{inklusionen} we see
$$\dczng\subset\cdczng\subset\rot\Hocgn(\Xi)\subset\rot\rcng\subset\dczng$$
and $\rot\rcng\subset\rot\crcng\subset\cdczng$,
completing the proof.
\end{proof}

\begin{rem}
\label{remrotpot}
Inspection of the above proof shows that the latter theorem holds for more general domains. 
Let $ \om\subset\rt $ be a bounded strong Lipschitz domain, such that $ \rt\setminus\ol\om $ is connected, 
and let $ \Gamma_\nu = \bigcup_{k=1}^{K} \Gamma_{\nu,k} $, $ K\in\nz $, 
with disjoint, relatively open and simply connected strong Lipschitz surface patches $ \Gamma_{\nu,k} \subset \Gamma$, 
where $ \dist (\Gamma_{\nu,k}, \Gamma_{\nu,\ell}) > 0 $ for all $ 1 \leq k \neq \ell \leq K $. 
Now extend $ \om $ over $ \Gamma_{\nu,k} $ by $ \hat \om_k $, let $ \tilde\om $ 
denote the interior of $ \ol\om\cup \ol{\hat\om}_1 \cup \dots \cup \ol{\hat\om}_K $ 
and define $ \tilde H $ like in \eqref{defHschlange}. Then $\tilde H\in\dz(\tilde \om) $.
Lemma \ref{div0rot} yields
$\T_{\mathsf{r}} \tilde H\in\Ho(\rt)\cap\dz(\rt)$
with $\rot(\T_{\mathsf{r}} \tilde H)=\tilde H$ in $\tilde \om$.
Again $ \rot(\T_{\mathsf{r}} \tilde H) = 0 $ in $ \hat \om_k $ for $ k = 1,\dots,K $.
Continuing analogously and since the $\hat\om_k$ are simply connected,
there exist unique potentials $\varphi_1,\dots, \varphi_K\in\Hoperp(\hat\om_{k})$ with
$\T_{\mathsf{r}} \tilde H=\na\varphi_k$ in $\hat\om_k$.
As before $ \varphi_k\in\Ht(\hat\om_k) $. 
Let $\E_k:\Ht(\hat\om_k)\rightarrow \Ht(\rt)$, $ k=1,\dots,K $, be extension operators. 
By cutting off appropriately it can be arranged that 
$ \supp(\E_k \varphi_k)\cap \ol{\hat\om}_\ell = \emptyset $ for all $ 1 \leq k \neq \ell \leq K $.
We define
\begin{align*}
\mathcal{S}_{\mathsf{r}} H:=\T_{\mathsf{r}} \tilde H - \sum_{k=1}^{K}\na(\E_k\varphi_k) \in\Ho(\rt).
\end{align*}
Again from $\mathcal{S}_{\mathsf{r}} H = 0$ in $\hat \om_k$, $ k=1,\dots,K $, 
$\mathcal{S}_{\mathsf{r}} H|_{\Gamma_{\nu}} = 0$ follows, 
which means $\mathcal{S}_{\mathsf{r}} H\in\Hocntr(\om)$ and therefore
$$\mathcal{S}_{\mathsf{r}} H \in\Hocn(\om)\subset\rcn(\om)\subset\crcnom.$$
Moreover $\rot(\mathcal{S}_{\mathsf{r}} H)=H$ in $\om$, 
as $\rot(\mathcal{S}_{\mathsf{r}} H) = \rot(\T_{\mathsf{r}} \tilde H) = \tilde H$ even in $\tilde \om$.
\end{rem}

\begin{theo}
\label{satzLtL6}
There exists a continuous linear operator
$$\mathcal{S}_{\mathsf{d}}:\LtXi\rightarrow\Ho(\rt)\cap\Hocgng,$$
such that for all $h\in\LtXi$
$$\div(\mathcal{S}_{\mathsf{d}} h)=h\quad\text{in }\Xi.$$
Especially $\LtXi=\div\Hocgn(\Xi)=\div\dcng=\div\cdcng$
and the $\Hocgn(\Xi)$-potential depends continuously on the data.
In particular these are closed subspaces of $\Ltom$.
\end{theo}

\begin{proof}
The case $\gamma_{\nu}=\emptyset$ immediately follows from Lemma \ref{lemL2divFT},
as for $h\in\LtXi$ we define
$$\mathcal{S}_{\mathsf{d}}h:=\T_{\mathsf{d}}h\in\Ho(\rt)\cap\rotspace_0(\rt).$$
The other two cases can again be handled together. 
Instead of extending $ \Xi $ over $ \gamma_\nu $ by a rectangle (as before), 
we extend it by a bubble in a way that $ \hat\Xi $ has a $ \Cont^3 $-boundary 
and $ \ol{\hat \Xi} \cap \ol \Xi = \ol{\gamma}_\nu $, cf. Figure \ref{muetze}. 
This smoothness of $ \hat\Xi $ allows for a later application of a standard Maxwell regularity result \cite{weberregmax}. 
Now let $h\in\LtXi$ and define $\tilde h\in\Lt(\tilde \Xi)$ by
\begin{align*}
\tilde h:=
\begin{cases}
h & \text{in }\Xi,\\
0 & \text{in }\hat \Xi.
\end{cases}
\end{align*}
Lemma \ref{lemL2divFT} yields
$$\T_{\mathsf{d}} \tilde h\in\Ho(\rt)\cap\rotspace_0(\rt)$$
with
\begin{align*}
\div(\T_{\mathsf{d}} \tilde h)=\tilde h~~\text{in}~\tilde \Xi.
\end{align*}
In particular  $\T_{\mathsf{d}} \tilde h\in\Ho(\hat \Xi)$ and $\div(\T_{\mathsf{d}} \tilde h) = 0$ in $\hat \Xi$.
Because $\rt\setminus\hat \Xi$ is connected, by Lemma \ref{div0rot} there exists 
a potential $\Phi\in\Ho(\rt) \cap \divspace_0(\rt)$ with
\begin{align*}
\T_{\mathsf{d}} \tilde h=\rot\Phi\quad\text{in}~\hat \Xi,
\end{align*}
so $\rot\Phi\in \Ho(\hat \Xi)$. Let $ \pi $
be the Helmholtz projector\footnote{For $F\in\Lt(\hat\Xi)$ solve by Lax-Milgram
$$\forall\,\varphi\in\Hoperp(\hat\Xi)\qquad\scp{\na u}{\na\varphi}_{\Lt(\hat\Xi)}=\scp{E}{\na\varphi}_{\Lt(\hat\Xi)}$$
with $u\in\Hoperp(\hat\Xi)$. Then the projector $\pi$ is given by
$\pi E:=E-\na u\in\big(\na\Hoperp(\hat\Xi)\big)^{\perp}$
since $\big(\na\Hoperp(\hat\Xi)\big)^{\perp}=\big(\na\Ho(\hat\Xi)\big)^{\perp}=\dcz(\hat \Xi)$.} 
onto solenoidal fields $ \dcz(\hat \Xi) $. With $ \Phi\in\Ho(\hat \Xi)\subset \rotspace(\hat \Xi) $ 
it follows $ \pi\Phi\in \rotspace(\hat \Xi)\cap\dcz(\hat \Xi) $ and $ \rot\pi\Phi = \rot\Phi \in\Ho(\hat \Xi) $. 
Thus $ \pi\Phi \in \Ht(\hat \Xi) $ by standard Maxwell regularity \cite{weberregmax}.
Let $\E:\Ht(\hat \Xi)\rightarrow \Ht(\rt)$ be a continuous, linear, e.g. Calderon's, extension operator. Define
\begin{align*}
\mathcal{S}_{\mathsf{d}} h:=\T_{\mathsf{d}} \tilde h-\rot(\E\pi\Phi) \in\Ho(\rt).
\end{align*}
Then $\mathcal{S}_{\mathsf{d}}:\LtXi\rightarrow\Ho(\rt)$ is linear and continuous.
With $ \mathcal{S}_{\mathsf{d}} h = 0 $ in $ \hat \Xi $ 
we see $ \mathcal{S}_{\mathsf{d}} h|_{\gamma_\nu} = 0 $, 
which means $\mathcal{S}_{\mathsf{d}} h\in\Hocgntr(\Xi)$ 
and with Lemma \ref{Hoct} $\mathcal{S}_{\mathsf{d}} h \in\Hocgn(\Xi)\subset\rcng\subset\crcng$. 
Moreover
$$\div(\mathcal{S}_{\mathsf{d}} h)=h\quad\text{in }\Xi,$$
as $\div(\mathcal{S}_{\mathsf{d}} h) = \div(\T_{\mathsf{d}} \tilde h) = \tilde h$ even in $\tilde \Xi$.
\end{proof}

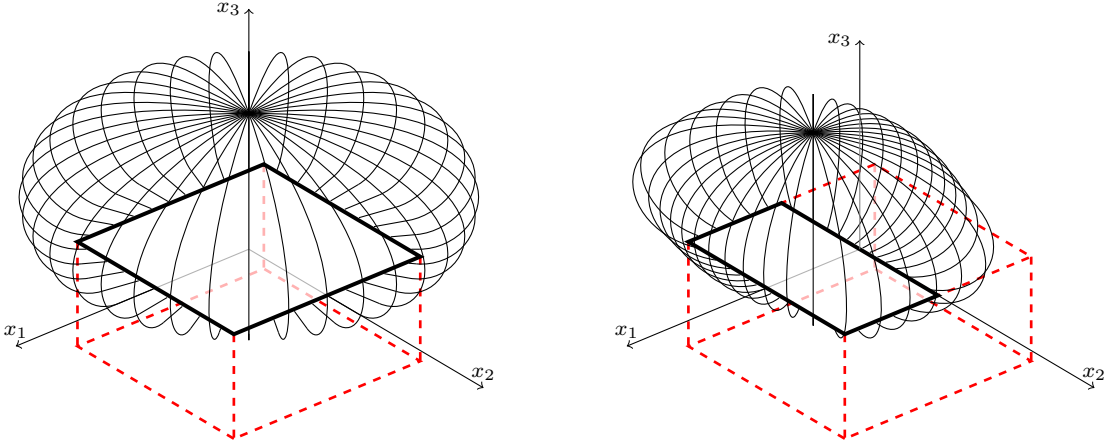
\begin{figure}
\centering
{\footnotesize		
\tdplotsetmaincoords{60}{140}
\begin{tikzpicture}[scale=1.6,tdplot_main_coords]
	\draw [black] (0,0,0) -- (1,0,0);
	\draw [black] (0,0,0) -- (0,1,0);	
	\draw [gray] (0,0,0) -- (0,0,1.3);
	\draw [red,dashed,line width=1pt] (1,1,-1) -- (1,-1,-1) -- (-1,-1,-1) -- (-1,1,-1) -- cycle;	
	\draw [red,dashed,line width=1pt] (1,1,0) -- (1,1,-1);							
	\draw [red,dashed,line width=1pt] (1,-1,0) -- (1,-1,-1);							
	\draw [red,dashed,line width=1pt] (-1,-1,0) -- (-1,-1,-1);						
	\draw [red,dashed,line width=1pt] (-1,1,0) -- (-1,1,-1);							
	\fill [white,opacity=.7] (1,1,0) -- (1,-1,0) -- (-1,-1,0) -- (-1,1,0) -- cycle;				
	\draw [black,line width=1.5pt] (1,1,0) -- (1,-1,0) -- (-1,-1,0) -- (-1,1,0) -- cycle;		
	\foreach \angle in {0,10,...,350}{
		\tdplotsetthetaplanecoords{\angle}			
		\draw [black,solid,tdplot_rotated_coords] (1.3,0) .. controls (1.3,2.5) and (0,2) .. (0,1.4);
	}
	\foreach \angle in {0,10,...,350}{
		\tdplotsetthetaplanecoords{\angle}			
		\ifthenelse{\not 0=\angle \and \not 180=\angle \and \not 90=\angle \and \not 270=\angle}	
		{ \def\endp{{min(abs(1/sin(\angle)),abs(1/cos(\angle)))}} }							
		{ \def\endp{1} }															
		\draw [black,solid,tdplot_rotated_coords] (0,1.4) -- (0,{(\endp)+0.03});
	}
	\draw [black,->] (1,0,0) -- (2.5,0,0) 	node[anchor=south]{$x_1$};
	\draw [black,->] (0,1,0) -- (0,3,0) 	node[anchor=south]{$x_2$};
	\draw [black,->] (0,0,1.3) -- (0,0,2.3) 	node[anchor=east]{$x_3$};
\end{tikzpicture}
\tdplotsetmaincoords{60}{140}
\begin{tikzpicture}[scale=1.6,tdplot_main_coords]
	\draw [black] (0,0,0) -- (1,0,0);
	\draw [black] (0,0,0) -- (0,1,0);
	\draw [red,dashed,line width=1pt] (1,1,-1) -- (1,-1,-1) -- (-1,-1,-1) -- (-1,1,-1) -- cycle;	
	\draw [red,dashed,line width=1pt] (1,1,0) -- (1,1,-1);							
	\draw [red,dashed,line width=1pt] (1,-1,0) -- (1,-1,-1);							
	\draw [red,dashed,line width=1pt] (-1,-1,0) -- (-1,-1,-1);						
	\draw [red,dashed,line width=1pt] (-1,1,0) -- (-1,1,-1);							
	\fill [white,opacity=.7] (1,1,0) -- (1,-1,0) -- (-1,-1,0) -- (-1,1,0) -- cycle;				
	\draw [red,dashed,line width=1pt] (0,-1,0) -- (0,1,0) -- (-1,1,0) -- (-1,-1,0) -- cycle; 	
	\draw [black,line width=1.5pt] (1,1,0) -- (1,-1,0) -- (0,-1,0) -- (0,1,0) -- cycle;		
	\coordinate (Orig) at (0.5,0,0);			
	\tdplotsetrotatedcoordsorigin{(Orig)}
	\foreach \angle in {0,10,...,350}{
		\tdplotsetthetaplanecoords{\angle}			
		\ifthenelse{\not 0=\angle \and \not 180=\angle \and \not 90=\angle \and \not 270=\angle}
		{
			\ifthenelse{0 < \angle \and \angle < 63}
			{ \def\endp{{min(abs(0.5/sin(\angle)),abs(0.5/cos(\angle)))}} }	{}
			\ifthenelse{63 < \angle \and \angle < 116}
			{ \def\endp{{min(abs(1/sin(\angle)),abs(1/cos(\angle)))}} }		{}
			\ifthenelse{116 < \angle \and \angle < 243}
			{ \def\endp{{min(abs(0.5/sin(\angle)),abs(0.5/cos(\angle)))}} }	{}
			\ifthenelse{243 < \angle \and \angle < 296}
			{ \def\endp{{min(abs(1/sin(\angle)),abs(1/cos(\angle)))}} }		{}
			\ifthenelse{296 < \angle \and \angle < 360}
			{ \def\endp{{min(abs(0.5/sin(\angle)),abs(0.5/cos(\angle)))}} }	{}
			\ifthenelse{\angle = 50}	{ \def\endp{{abs(0.5/cos(\angle))}} }		{}
			\ifthenelse{\angle = 60}	{ \def\endp{{abs(0.5/cos(\angle))}} }		{}
			\ifthenelse{\angle = 120}	{ \def\endp{{abs(0.5/cos(\angle))}} }		{}
			\ifthenelse{\angle = 130}	{ \def\endp{{abs(0.5/cos(\angle))}} }		{}
			\ifthenelse{\angle = 230}	{ \def\endp{{abs(0.5/cos(\angle))}} }		{}
			\ifthenelse{\angle = 240}	{ \def\endp{{abs(0.5/cos(\angle))}} }		{}
			\ifthenelse{\angle = 300}	{ \def\endp{{abs(0.5/cos(\angle))}} }		{}
			\ifthenelse{\angle = 310}	{ \def\endp{{abs(0.5/cos(\angle))}} }		{}
		}{}
		\ifthenelse{\angle=0}		{ \def\endp{0.5} }	{}
		\ifthenelse{\angle=180}	{ \def\endp{0.5} }	{}
		\ifthenelse{\angle=90	}	{ \def\endp{1} }		{}
		\ifthenelse{\angle=270}	{ \def\endp{1} }		{}
		\draw [black,solid,tdplot_rotated_coords] (1.3,0) .. controls (1.3,{(\endp)+1}) and (0.05,{(\endp)+0.8}) .. (0,{(\endp)+0.03});
	}
	\draw [black,->] (1,0,0) -- (2.5,0,0) 	node[anchor=south]{$x_1$};
	\draw [black,->] (0,1,0) -- (0,3,0) 	node[anchor=south]{$x_2$};
	\draw [gray] (0,0,0) -- (0,0,1.3);
	\draw [black,->] (0,0,1.3) -- (0,0,2) 	node[anchor=east]{$x_3$};
\end{tikzpicture}
}
\caption{The half-cube $\Xi=B_{-}$, extended by a $\Cont^3$-domain $\hat\Xi$
to $ \tilde \Xi $, and the rectangles $\gamma_{\nu}=B_{0}$ and $\gamma_{\nu}=B_{0,+}$.} 
\label{muetze}
\end{figure}

\begin{rem}
\label{remdivpot}
Theorem \ref{satzLtL6} again holds for more general domains $\om\subset\rt$.
For example, $\om$ can be a bounded strong Lipschitz domain with 
$$ \Gamma_\nu = \overset{K}{\underset{k=1}{\dot\bigcup}}\Gamma_{\nu,k},\quad K\in\nz,$$
where $ \dist (\Gamma_{\nu,k}, \Gamma_{\nu,\ell}) > 0 $ for all $ 1 \leq k \neq \ell \leq K $
and $\Gamma_{\nu,k}$ are $\Cont^{3}$-boundary patches allowing 
for $\Cont^{3}$-regular extensions $\hat\om_{\nu,k}$
having connected complements $\rt\setminus\ol{\hat\om}_{\nu,k}$.
If $\Gamma_{\nu}=\Gamma$, then the right hand side $h$ must have vanishing mean value,
i.e., $h\in\Ltperpom$.
\end{rem}

We note that in the case of $\gamma_{\nu}=\gamma$ resp. $\Gamma_{\nu}=\Gamma$  
Theorem \ref{satzLtL6} resp. Remark \ref{remdivpot}
is a well known result for bounded Lipschitz domains\footnote{Again, in this case 
$h$ must have vanishing mean value.}.
An elegant proof can be found in \cite[Lemma 2.1.1]{sohrbook}.
See also \cite[Lemma 3.2]{bauerneffpaulystarkedevDivsymCurl} for more recent results
including mixed boundary conditions, where it has been shown 
that Theorem \ref{satzLtL6} resp. Remark \ref{remdivpot} even holds for general bounded strong Lipschitz pairs 
$(\Xi,\gamma_{\nu})$ resp. $(\om,\Gamma_{\nu})$.

\subsection{Weak equals strong for the half-cube in terms of boundary conditions}

Now the two main density results immediately follow.
We note that this has already been proved for the $\Hoom$-spaces in Lemma \ref{Hoct}, i.e., $\cHoctom=\Hoctom$.

\begin{theo}
\label{rss}
$\crcng=\rcng$ and $\cdcng=\dcng$.
\end{theo}

\begin{proof}
Suppose $E\in\crcng$ and thus $\rot E\in\cdcztg$. By Theorem \ref{satzD0L6} there exists 
$\hat E\in\rcng$ with $\rot\hat E = \rot E$. 
By Theorem \ref{gradphirot0} we get $E-\hat E\in\crczng=\rczng$
and hence $E\in\rcng$.
Analogously let $H\in\cdcng$ and thus $\div H\in\Lt(\Xi)$.
By Theorem \ref{satzLtL6} there exists 
$\hat H\in\dcng$ with $\div\hat H = \div H$. 
By Theorem \ref{satzD0L6} we get $H-\hat H\in\cdczng=\dczng$
and hence $H\in\dcng$.
\end{proof}

\section{The compact embedding}

\subsection{Compact embedding on the half-cube}

First we show the main result on the half-cube $\Xi=B_{-}$ with the special boundary patch
$$\gamma_{\nu}=\emptyset,\quad\gamma_{\nu}=B_{0}\quad\text{or}\quad\gamma_{\nu}=B_{0,+}$$
from the latter section. For this let $\eps\in\Li(\Xi)$ be an admissible matrix field.

\begin{theo}
\label{HSG}
The embedding $\rctg\cap\eps^{-1}\dcng \hookrightarrow \Lt(\Xi)$ is compact.
\end{theo}

\begin{proof}
The cases of full boundary conditions, i.e., $\gamma_{\tau}=\gamma$ or $\gamma_{\tau}=\emptyset$,
are well known, see the introduction. 
Suppose $\gamma_{\nu}=B_{0}$ or $\gamma_{\nu}=B_{0,+}$.
Let $(H_n)_{n\in\nz}$ be a bounded sequence in $\rctg\cap\eps^{-1}\dcng$. 
By Riesz' representation theorem\footnote{We equip $ \Hocgt(\Xi) $ with 
the scalar product $ \scpLtXi{\eps\na\,\cdot\,}{\na\,\cdot\,} .$}, 
for all $n\in\nz$ there exists a unique $u_n\in\Hocgt(\Xi)$ with
\begin{align}
\label{psindef}
\scpLtXi{\eps\na u_n}{\na\varphi} = \scpLtXi{\eps H_n}{\na\varphi}~\text{for all}~\varphi\in\Hocgt(\Xi).
\end{align}
Furthermore $\norms{u_n}_{\Ho(\Xi)} \leq c\norms{H_n}_{\LtXi}$
and w.l.o.g.~by Rellich's selection theorem $(u_{n})$ converges in $\LtXi$.
By definition, \eqref{psindef} together with Theorem \ref{satzD0L6} 
implies $$H_n-\na u_n\in\eps^{-1}\cdczng = \eps^{-1} \dczng$$ 
and  $\na u_n\in\rcztg$ by Theorem \ref{gradphirot0}, so 
$$\tilde H_n :=H_n-\na u_n \in \rctg\cap\eps^{-1}\dczng.$$
Now we apply Theorem \ref{satzD0L6} to $\eps\tilde H_n$ and define 
$E_n:=\mathcal{S}_{\mathsf{r}}\eps\tilde H_n\in\Hocgn(\Xi)$, which satisfies
$$\norms{E_n}_{\Ho(\Xi)}
\leq c\norms{\tilde H_n}_{\LtXi}
\leq c\norms{H_n}_{\LtXi}.$$
W.l.o.g.~by Theorem \ref{satzD0L6} and Rellich's selection theorem $(E_{n})$ converges in $\LtXi$.
Moreover, we observe $\rot E_n = \eps \tilde H_n$ and thus
\begin{align*}
\normsLtg{\sqrt{\eps}(\tilde H_n- \tilde H_m)}^2
&=\scpLtXi{\tilde H_n- \tilde H_m}{\rot(E_n-E_m)}\\
&=\scpLtXi{\rot(\tilde H_n- \tilde H_m)}{E_n-E_m}
\leq c\normsLtg{E_n-E_m},
\end{align*}
as $\rot \tilde H_n = \rot H_n$. Thus $(\tilde H_n)$ converges in $\LtXi$. 
Moreover, \eqref{psindef} yields
\begin{align*}
\normsLtg{\sqrt{\eps}\na(u_n-u_m)}^2
&= \scpLtXi{\eps(H_n-H_m)}{\na(u_n-u_m)}\\
&=-\scpLtXi{\div(\eps(H_n-H_m))}{u_n-u_m}
\leq c\normsLtg{u_n-u_m}
\end{align*}
and hence also $(\na u_{n})$ converges in $\LtXi$, i.e., $(u_{n})$ converges in $\Ho(\Xi)$.
Therefore, $(H_{n})$ converges in $\LtXi$.
\end{proof}

\subsection{The compact embedding for weak Lipschitz domains}
\label{mcpweaklip}

The aim of this section is to transfer Theorem \ref{HSG} to arbitrary weak Lipschitz 
pairs $(\om,\Gamma_{\tau})$. 
We need the technical Lemma \ref{lemtrafo}, 
for a proof see \cite[Section 3]{picardcomimb} or \cite[Remark 2]{wecktrace}.
Let us consider the following situation:
Let $\Theta$, $\tilde{\Theta}$ be two domains in $\rt$
with boundaries $\Upsilon:=\p\Theta$, $\tilde\Upsilon:=\p\tilde\Theta$
and $\Upsilon_0\subset\Upsilon$, let's say, relatively open.
Moreover, let 
$$\phi:\Theta\to\tilde\Theta,\qquad
\psi:=\phi^{-1}:\tilde\Theta\to\Theta$$
be Lipschitz diffeomorphisms, this is,
$\phi\in\Czo(\Theta,\tilde\Theta)$ and $\psi=\phi^{-1}\in\Czo(\tilde\Theta,\Theta)$. 
Hence there exists a constant $c$, such that for all $x\in\Theta$ and $\tilde x\in\tilde\Theta$
$$0<c\leq\big|\det\phi'(x)\big|,\big|\det\psi'(\tilde x)\big|\leq1/c.$$
Then $\tilde\Theta=\phi(\Theta)$, $\tilde\Upsilon=\phi(\Upsilon)$
and we define $\tilde\Upsilon_0:=\phi(\Upsilon_0)$.
To simplify the notations here and throughout this section and the appendix
we will use the notation 
$$\tilde u:=u\circ\psi,\qquad
\utilde v:=v\circ\phi$$
both for functions and for vector fields. We set, identify and note
$$J:=J_{\psi}=\psi'\in\Li(\tilde\Theta),\qquad
\phi'=J_{\phi}=J^{-1}\circ\phi=\utilde J{}^{-1}\in\Li(\Theta).$$

\begin{lem}
\label{lemtrafo}
Let $u\in\HoctUT$, 
$E\in\crotgen{\Upsilon_0}{\circ}(\Theta)$ resp. $E\in\rotgen{\Upsilon_0}{\circ}(\Theta)$ 
and $H\in\cdivgen{\Upsilon_0}{\circ}(\Theta)$ resp. $H\in\divgen{\Upsilon_0}{\circ}(\Theta)$. Then
\begin{align*}
\tilde u&\in\HoctUTt&
&\text{and}&
\na\tilde u&=J^{\top}\widetilde{\na u},\\
J^{\top}\tilde E&\in\crotgen{\tilde\Upsilon_0}{\circ}(\tilde\Theta)
\text{ resp. }\rotgen{\tilde\Upsilon_0}{\circ}(\tilde\Theta)&
&\text{and}&
\rot(J^{\top}\tilde E)&=(\det J)J^{-1}\widetilde{\rot E},\\
(\det J)J^{-1}\tilde H&\in\cdivgen{\tilde\Upsilon_0}{\circ}(\tilde\Theta)
\text{ resp. }\divgen{\tilde\Upsilon_0}{\circ}(\tilde\Theta)&
&\text{and}&
\div((\det J)J^{-1}\tilde H)&=\det J\widetilde{\div H}.
\end{align*}
\end{lem}

For a rigorous proof of the latter lemma see the appendix of our contributions 
\cite{bauerpaulyschomburgmaxcompweakliprNarxiv,bauerpaulyschomburgmaxcompweakliprN}.

From now on, we make the following

\vspace*{4mm}
\noindent
{\bf General Assumption:} Let $(\om,\Gamma_{\tau})$
be a weak Lipschitz pair as in Definitions \ref{defilipmani} and \ref{defilipsubmani}. 
\vspace*{4mm}

We adjust Lemma \ref{lemtrafo} to our situation:
Let $U_1,\dots,U_K$ be an open covering of $\Gamma$
according to Definitions \ref{defilipmani} and \ref{defilipsubmani}
and set $U_{0}:=\om$. Therefore $U_0,\dots,U_K$ is an open covering of $\ol\om$.
Moreover let $\chi_k\in\Cic(U_k)$, $k\in\{0,\dots,K\}$, 
be a partition of unity subordinate to the open covering $U_0,\dots,U_K$.
Now suppose $k\in\{1,\dots,K\}$. We define
\begin{align*}
\om_k&:=U_k\cap\om,&
\Gamma_k&:=U_k\cap\Gamma,&
\Gamma_{\tau,k}&:=U_k\cap\Gamma_{\tau},&
\Gamma_{\nu,k}&:=U_k\cap\Gamma_{\nu},\\
\hat\Gamma_{k}&:=\p\om_{k},&
\Sigma_k&:=\hat\Gamma_{k}\setminus\Gamma,&
\hat\Gamma_{\tau,k}&:=\textrm{int}(\Gamma_{\tau,k}\cup\ol\Sigma_k),&
\hat\Gamma_{\nu,k}&:=\textrm{int}(\Gamma_{\nu,k}\cup\ol\Sigma_k),\\
&&
\sigma&:=\gamma\setminus\ol B_{0},&
\hat\gamma_{\tau}&:=\textrm{int}(\gamma_{\tau}\cup\ol\sigma),&
\hat\gamma_{\nu}&:=\textrm{int}(\gamma_{\nu}\cup\ol\sigma).
\end{align*}
Lemma \ref{lemtrafo} will from now on be used with 
$$\Theta:=\om_k,\quad
\tilde\Theta:=\Xi,\qquad
\phi:=\phi_k:\om_{k}\to\Xi,\quad
\psi:=\psi_k:\Xi\to\om_{k}$$
and with one of the following cases
$$\Upsilon_0:=\Gamma_{\tau,k},\quad
\Upsilon_0:=\hat\Gamma_{\tau,k},\quad
\Upsilon_0:=\Gamma_{\nu,k}\quad\text{or}\quad
\Upsilon_0:=\hat\Gamma_{\nu,k}.$$
Then $\Upsilon=\hat\Gamma_{k}$ and $\tilde\Upsilon=\phi_{k}(\hat\Gamma_{k})=\gamma$ 
as well as (depending on the respective case)
\begin{align*}
\tilde\Upsilon_0
&=\phi_{k}(\Gamma_{\tau,k})
=\gamma_{\tau},&
\tilde\Upsilon_0
&=\phi_{k}(\hat\Gamma_{\tau,k})
=\hat\gamma_{\tau},&
\gamma_{\tau}
&\in\{\emptyset,B_{0},B_{0,-}\},&
\gamma_{\nu}
&=\gamma\setminus\ol\gamma_{\tau},\\
\tilde\Upsilon_0
&=\phi_{k}(\Gamma_{\nu,k})
=\gamma_{\nu},&
\tilde\Upsilon_0
&=\phi_{k}(\hat\Gamma_{\nu,k})
=\hat\gamma_{\nu},&
\gamma_{\nu}
&\in\{\emptyset,B_{0},B_{0,+}\},&
\gamma_{\tau}
&=\gamma\setminus\ol\gamma_{\nu}.
\end{align*}

\begin{rem}
\label{Bminus}
Theorems \ref{gradphirot0}, \ref{satzD0L6}, \ref{satzLtL6}
and Remarks \ref{gradphirot0rem}, \ref{remrotpot}, \ref{remdivpot} 
as well as Theorems \ref{rss}, \ref{HSG}
hold for $\gamma_{\nu}=B_{0,-}$ without any (substantial) modification as well.
\end{rem}

\begin{lem}
\label{kor320}
Let $k\in\{1,\dots,K\}$. For $E\in\crctom$ and $H\in\cdcnom$ we have
\begin{align*}
E\in\crotgen{\Gamma_{\tau,k}}{\circ}(\om_k),\qquad
\chi_k E\in\crcthat(\om_k),\qquad
H\in\cdivgen{\Gamma_{\nu,k}}{\circ}(\om_k),\qquad
\chi_k H\in\cdcnhat(\om_k).
\end{align*}
\end{lem}

\begin{proof}
Let $\Phi\in\Cic_{\hat\Gamma_{\nu,k}}(\om_k)$. 
Extending $\Phi$ by zero shows $\Phi\in\Cicn(\om)$ and
\begin{align*}
\scp{E}{\rot\Phi}_{\Lt(\om_k)}
=\scpLtom{E}{\rot\Phi}
=\scpLtom{\rot E}{\Phi}
=\scp{\rot E}{\Phi}_{\Lt(\om_k)},
\end{align*}
hence $E\in\crotgen{\Gamma_{\tau,k}}{\circ}(\om_k)$. 
Let $\Phi\in\Cic_{\Gamma_{\nu,k}}(\om_k)$. 
Then, $\chi_k\Phi\in\Cic_{\hat\Gamma_{\nu,k}}(\om_k)\subset\Cicn(\om)$
since $\supp \chi_k\subset U_k$ and
\begin{align*}
\scp{\chi_k E}{\rot\Phi}_{\Lt(\om_k)}
&=\scpLtom{\chi_k E}{\rot\Phi}
=\scpLtom{E}{\rot(\chi_k\Phi)}-\scpLtom{E}{\na\chi_k\times\Phi}\\
&=\scpLtom{\rot E}{\chi_k\Phi}+\scpLtom{\na\chi_k\times E}{\Phi}
=\scp{\rot(\chi_k E)}{\Phi}_{\Lt(\om_k)},
\end{align*}
thus $\chi_k E\in\crcthat(\om_k)$. 
Analogously we see $H\in\cdivgen{\Gamma_{\nu,k}}{\circ}(\om_k)$
and $\chi_k H\in\cdcnhat(\om_k)$.
\end{proof}

\begin{theo}
\label{theoweakeqstrong}
$\crctom=\rctom$ and $\cdcnom=\dcnom.$
\end{theo}

\begin{proof}
Suppose $E\in\crctom$. Then $\chi_{0}E\in\rcom\subset\rctom$ by mollification.
Let $k\in\{1,\dots,K\}$. Then $E\in\crotgen{\Gamma_{\tau,k}}{\circ}(\om_k)$ by Lemma \ref{kor320}. 
Lemma \ref{lemtrafo}, Theorem \ref{rss} (with $\gamma_{\nu}:=\gamma_{\tau}$) 
and Remark \ref{Bminus} yield with $J_{k}:=J_{\psi_{k}}$
$$J^{\top}_k\tilde E\in\crctg=\rctg,\qquad
\gamma_{\tau}=\phi_{k}(\Gamma_{\tau,k})\in\{\emptyset,B_{0},B_{0,-}\}.$$
Hence, by definition there exists a sequence $(\tilde A_{k,\ell})\subset\Czoc_{\gamma_{\tau}}(\Xi)$ 
with $\tilde A_{k,\ell}\xrightarrow{\ell\to\infty}J^{\top}_k\tilde E$ in $\rotspace(\Xi)$.
Therefore $\tilde\chi_{k}\tilde A_{k,\ell}\in\Czoc_{\hat\gamma_{\tau}}(\Xi)$. 
Now define with Lemma \ref{lemtrafo}
$$E_{k,\ell}:=\chi_{k}\utilde J{}^{-\top}_k A_{k,\ell}
=(\tilde\chi_{k}J^{-\top}_k\tilde A_{k,\ell})\circ \phi_k
\in\overset{\circ}{\rotspace}_{\hat\Gamma_{\tau,k}}(\om_k)\subset\rctom.$$
Then with Lemma \ref{lemtrafo}
\begin{align*}
\norms{E_{k,\ell}-\chi_k E}^2_{\rom} 
&=\norms{E_{k,\ell}-\chi_k E}^2_{\rotspace(\om_k)}
=\int_{\om_k}\big(\normb{E_{k,\ell}-\chi_k E}^2
+\normb{\rot(E_{k,\ell}-\chi_k E)}^2\big)\\
&=\int_{\Xi}\norm{\det J_k}\Big(\normb{\tilde\chi_{k}(J^{-\top}_k\tilde A_{k,\ell}-\tilde E)}^2
+\normb{\widetilde{\rot\big(\chi_{k}(\utilde J{}^{-\top}_k A_{k,\ell}-E)\big)}}^2\Big)\\
&=\int_{\Xi}\norm{\det J_k}\Big(\normb{\tilde\chi_{k}(J^{-\top}_k\tilde A_{k,\ell}-\tilde E)}^2
+\normb{(\det J_k)^{-1}J_k\rot(J_{k}^{\top}\tilde\chi_{k}(J^{-\top}_k \tilde A_{k,\ell}-\tilde E)}^2\Big)\\
&\leq c\int_{\Xi}\Big(\normb{\tilde\chi_{k}(\tilde A_{k,\ell}-J^{\top}_k\tilde E)}^2
+\normb{\rot(\tilde\chi_{k}(\tilde A_{k,\ell}-J_{k}^{\top}\tilde E)}^2\Big)\\
&\leq c\int_{\Xi}\Big(\normb{\tilde A_{k,\ell}-J^{\top}_k\tilde E}^2
+\normb{\rot(\tilde A_{k,\ell}-J_{k}^{\top}\tilde E)}^2\Big)
=c\norms{\tilde A_{k,\ell}-J_{k}^{\top}\tilde E}_{\rotspace(\Xi)}^2
\xrightarrow{\ell\to\infty}0,
\end{align*}
i.e., $\chi_k E\in\rctom$ by Remark \ref{remlipclosure} (i), and thus $E=\sum_{k}\chi_k E\in\rctom$.
$\cdcnom=\dcnom$ is proved analogously for $H\in\cdcnom$ by approximating
$$(\det J_{k})J_{k}^{-1}\tilde H\in\cdcng=\dctg,\qquad
\gamma_{\nu}=\phi_{k}(\Gamma_{\nu,k})\in\{\emptyset,B_{0},B_{0,+}\},$$
with a sequence $(\tilde A_{k,\ell})\subset\Czoc_{\gamma_{\nu}}(\Xi)$ in $\divspace(\Xi)$.
\end{proof}

\begin{rem}
\label{remssabsch}
By Theorem \ref{theoweakeqstrong},
Lemma \ref{kor320} also holds for the spaces $\rctom$ and $\dcnom$.
More precisely, for $E\in\rctom$ and $H\in\dcnom$ we have for $k\in\{1,\dots,K\}$
\begin{align*}
E\in\rotgen{\Gamma_{\tau,k}}{\circ}(\om_k),\qquad
\chi_k E\in\rcthat(\om_k),\qquad
H\in\divgen{\Gamma_{\nu,k}}{\circ}(\om_k),\qquad
\chi_k H\in\dcnhat(\om_k).
\end{align*}
\end{rem}

Now the compact embedding for weak Lipschitz pairs $(\om,\Gamma_{\tau})$ can be proved.

\begin{theo}
\label{satzMKE}
Let $\eps\in\Li(\om)$ be an admissible matrix field.
Then the embedding
$$\rctom\cap\eps^{-1}\dcnom\hookrightarrow\Ltom$$
is compact.
\end{theo}

\begin{proof}
Suppose $(E_n)$ is a bounded sequence in $\rctom\cap\eps^{-1}\dcnom$. Then by mollification
$$E_{0,n}:=\chi_0E_n\in\rcom\cap\eps^{-1}\dcom,$$
$E_{0,n}$ even has compact support in $\om$,
and by classical results, see our introduction, $(E_{0,n})$ contains a $\Ltom$-converging subsequence,
again denoted by $(E_{0,n})$. Hence $E_{0,n}\to E_{0}$ in $\Ltom$ with some $E_{0}\in\Ltom$.
Let $k\in\{1,\dots,K\}$. By Lemma \ref{kor320} and Remark \ref{remssabsch}
$$E_{k,n}:=\chi_kE_n\in\rcthat(\om_k),\qquad
\eps E_{k,n}\in\dcnhat(\om_k)$$
and with
\begin{align*}
\rot E_{k,n}&=\chi_k\rot E_n+\na\chi_k\times E_n,&
\div(\eps E_{k,n})&=\chi_k\div(\eps E_n)+\na\chi_k\cdot\eps E_n
\end{align*}
the sequence $(E_{k,n})$ is bounded in $\rcthat(\om_k)\cap\eps^{-1}\dcnhat(\om_k)$. 
We define and see by Lemma \ref{lemtrafo}
$$\hat E_{k,n}:=J^{\top}_k\tilde E_{k,n}\in\rotgen{\hat\gamma_{\tau}}{\circ}(\Xi),\qquad
\rot\hat E_{k,n}=(\det J_k)J^{-1}_k\widetilde{\rot E_{k,n}}.$$
Hence
\begin{align*}
\norms{\hat E_{k,n}}^2_{\rotspace(\Xi)} 
&=\int_{\Xi}\big(|\hat E_{k,n}|^2 + |\rot\hat E_{k,n}|^2\big)
=\int_{\om_k}|\det\utilde J{}_k|^{-1} 
\Big(\big|\utilde J{}^{\top}_k E_{k,n}\big|^2 
+\big|(\det\utilde J{}_k)\utilde J{}^{-1}_k\rot E_{k,n}\big|^2\Big)\\
&\leq c\int_{\om_k}\big(|E_{k,n}|^2 
+|\rot E_{k,n}|^2\big)
=c\,\norms{E_{k,n}}^2_{\rotspace(\om_k)},
\end{align*}
showing that $(\hat E_{k,n})$ is bounded in $\rotgen{\hat\gamma_{\tau}}{\circ}(\Xi)$.
From
$$\hat E_{k,n}
=J^{\top}_k\tilde E_{k,n}
=(\det J_k)^{-1} J^{\top}_k\tilde\eps^{-1}J_k(\det J_k)J^{-1}_k\tilde\eps\tilde E_{k,n}$$
we observe by Lemma \ref{lemtrafo}
$$\hat\eps_k\hat E_{k,n} 
=(\det J_k)J^{-1}_k\tilde\eps\tilde E_{k,n}\in\divgen{\hat\gamma_{\nu}}{\circ}(\Xi),\quad
\hat\eps_k:=(\det J_k) J_k^{-1}\tilde\eps J^{-\top}_k,\quad
\div\hat\eps_k\hat E_{k,n}=(\det J_k)\widetilde{\div(\eps E_{k,n})}.$$
$\hat\eps_k$ is admissible, as $\hat\eps_k\in\Li(\Xi)$ and for all $H\in\Lt(\Xi)$
$$\scp{\hat\eps_kH}{H}_{\Lt(\Xi)}
=\scp{(\det J_k)J^{-1}_k\tilde\eps J^{-\top}_kH}{H}_{\Lt(\Xi)}
\geq c\,\norms{J^{-\top}_kH}^2_{\Lt(\Xi)}
\geq c\,\norms{H}_{\Lt(\Xi)}^2.$$
Then
\begin{align*}
\norms{\hat E_{k,n}}^2_{\hat\eps^{-1}_k\divspace(\Xi)}
&=\int_\Xi(|\hat E_{k,n}|^2 
+|\div\hat\eps_k\hat E_{k,n}|^2)
=\int_{\om_k}|\det\utilde J{}_k|^{-1} 
\Big(\big|\utilde J{}^{\top}_k E_{k,n}\big|^2 
+\big|(\det\utilde J{}_k)\div(\eps E_{k,n})\big|^2\Big)\\
&\leq c\int_{\om_k}\big(|E_{k,n}|^2 
+|\div(\eps E_{k,n})|^2\big)
=c\,\norms{E_{k,n}}^2_{\eps^{-1}\divspace(\om_k)}
\end{align*}
shows that $(\hat E_{k,n})$ is bounded in $\hat\eps^{-1}_k\divgen{\hat\gamma_{\nu}}{\circ}(\Xi)$.
Thus $(\hat E_{k,n})$ is bounded in 
$$\rotgen{\hat\gamma_{\tau}}{\circ}(\Xi)\cap\hat\eps^{-1}_k\divgen{\hat\gamma_{\nu}}{\circ}(\Xi)
\subset\rotgen{\hat\gamma_{\tau}}{\circ}(\Xi)\cap\hat\eps^{-1}_k\dcng,\quad
\gamma_{\nu}\in\{\emptyset,B_{0},B_{0,+}\},\quad
\hat\gamma_{\tau}=\gamma\setminus\ol\gamma_{\nu}.$$
Therefore, w.l.o.g. $\hat E_{k,n}\xrightarrow{n\to\infty}\hat E_k$ in $\Lt(\Xi)$
with some $\hat E_k\in\Lt(\Xi)$ by Theorem \ref{HSG}. Let
$$E_{k,n}:=\utilde J{}^{-\top}_k\utilde{\hat E}{}_{k,n}\in\Lt(\om_{k}),\qquad
E_{k}:=\utilde J{}^{-\top}_k\utilde{\hat E}{}_{k}\in\Lt(\om_{k})$$
and derive
\begin{align*}
\norms{E_{k,n}-E_{k}}^2_{\Lt(\om_k)} 
&=\int_{\Xi}\norm{\det J_k}|\tilde E_{k,n}-\tilde E_{k}|^2
=\int_{\Xi} \norm{\det J_k} |J^{-\top}_k\hat E_{k,n}-J^{-\top}_k\hat E_{k}|^2 \\
&\leq c\int_{\Xi}|\hat E_{k,n}-\hat E_{k}|^2
=c\,\norms{\hat E_{k,n}-\hat E_{k}}^2_{\Lt(\Xi)}.
\end{align*}
Hence $E_{k,n}\xrightarrow{n\to\infty}E_{k}$ in $\Lt(\om_k)$ and 
$E_{k,n}\xrightarrow{n\to\infty}E_{k}$ in $\Ltom$ for their extensions by zero to $\om$.
Finally 
$E_n=\sum_{k}\chi_kE_n=\sum_{k}E_{k,n}\xrightarrow{n\to\infty}\sum_{k}E_{k}$ in $\Ltom$.
\end{proof}

By the same but much simpler arguments
Rellich's selection theorem holds for weak Lipschitz domains $\om$
resp. weak Lipschitz pairs $(\om,\Gamma_{\tau})$.
Therefore, also Poincar\'e's estimate holds in this case
by a standard indirect argument.

\begin{theo}
\label{rellich}
Rellich's selection theorem and the Poincar\'e-Friedrichs estimate hold. More precisely:
\begin{itemize}
\item[\bf(i)] The embedding $\Hoctom\hookrightarrow\Ltom$ is compact.
\item[\bf(ii)] There exists a constant $c_{\mathsf{p}}>0$, 
such that $\norms{u}_{\Ltom}\leq c_{\mathsf{p}}\,\normsLtom{\na u}$ holds for all $u\in\Hoctom$
or all $u\in\Hoperpom$, if $\Gamma_\tau=\emptyset$.
\end{itemize}
\end{theo}

\section{Applications}

From now on let $\om\subset\rt$ be a bounded domain and
let $(\om,\Gamma_{\tau})$ be a weak Lipschitz pair as well as $\eps\in\Liom$ be admissible. 

\subsection{The Maxwell estimate}

A first consequence of the compact embedding Theorem \ref{satzMKE}, i.e.,
$$\rctom\cap\eps^{-1}\dcnom\hookrightarrow\Ltom$$
is that the space of so-called `Dirichlet-Neumann fields'
$$\harmdieps:=\rcztom\cap\eps^{-1}\dcznom$$
is finite dimensional because the unit ball in $\harmdieps$ is compact.
By a standard indirect argument Theorem \ref{satzMKE} immediately implies the so-called Maxwell estimate:

\begin{theo}
\label{MA}
There exists a constant $c_{\mathsf{m}}>0$, 
such that for all $E\in\rctom\cap\eps^{-1}\dcnom\cap\harmdieps^{\perp_\eps}$
\begin{align*}
\norms{E}_{\Ltepsom}
\leq c_{\mathsf{m}}\,\big(\normsLtom{\rot E}^2
+\normsLtom{\div(\eps E)}^2\big)^{1/2}.
\end{align*}
\end{theo}

\vspace*{1mm}
Here we introduce $\Ltepsom:=\Ltom$ equipped with the scalar product 
$\scp{\,\cdot\,}{\,\cdot\,}_{\Ltepsom}:=\scpLtom{\eps\,\cdot\,}{\,\cdot\,}$.

\begin{proof}
Suppose the estimate does not hold. Then there exists a sequence 
$$(E_{n})\subset\rctom\cap\eps^{-1}\dcnom\cap\harmdieps^{\perp_\eps}$$
with $\norms{E_{n}}_{\Ltepsom}=1$ and 
$$\normsLtom{\rot E_{n}}+\normsLtom{\div(\eps E_{n})}\to0.$$
Since $(E_{n})$ is bounded in $\rom\cap\eps^{-1}\dom$,
by Theorem \ref{satzMKE} there exists a $\Ltom$-converging
subsequence, again denoted by $(E_{n})$, with $E_{n}\to E\in\Ltom$.
But then $E_{n}\to E$ in $\rom\cap\eps^{-1}\dom$ and we see
$E\in\rzom\cap\eps^{-1}\dzom$.
As $\rctom\cap\eps^{-1}\dcnom\cap\harmdieps^{\perp_\eps}$ 
is a closed subspace of $\rom\cap\eps^{-1}\dom$,
we get 
$$E\in\rcztom\cap\eps^{-1}\dcznom\cap\harmdieps^{\perp_\eps}
=\harmdieps\cap\harmdieps^{\perp_\eps}
=\{0\},$$ 
a contradiction to $1=\norms{E_{n}}_{\Ltepsom}\to\norms{E}_{\Ltepsom}=0$.
\end{proof}

\subsection{Helmholtz decompositions}

Applying the projection theorem to the linear and closed operator 
$\na_{\tau}:\Hoctom\subset\Ltom\rightarrow\Ltepsom$ with adjoint 
$-\div_{\nu}\eps=\na_{\tau}^{*}:\eps^{-1}\cdcnom\subset\Ltepsom\rightarrow\Ltom$ yields
\begin{align}
\label{hzgrad}
\Ltom=\overline{\na\Hoctom}\oplus_\eps\eps^{-1}\cdcznom,
\end{align}
and $\cdcznom=\dcznom$ follows by Theorem \ref{theoweakeqstrong}. 
Here $\oplus_\eps$ denotes the orthogonal sum in $\Ltepsom$.
On the other hand, for the closed linear operator $\eps^{-1}\rot_{\nu}:\rcnom\subset\Ltom\rightarrow\Ltepsom$
with adjoint $\rot_{\tau}:=(\eps^{-1}\rot_{\nu})^{*}:\crctom\subset\Ltepsom\rightarrow\Ltom$ we get 
\begin{align}
\label{hzrot}
\Ltom=\crcztom\oplus_\eps\eps^{-1}\overline{\rot\rcnom},
\end{align}
and again $\crcztom=\rcztom$ by Theorem \ref{theoweakeqstrong}.
Since
$$\na\Hoctom\subset\rcztom,\qquad
\rot\rcnom\subset\dcznom,$$
\eqref{hzgrad} and \eqref{hzrot} yield
$$\rcztom
=\overline{\na\Hoctom}\oplus_\eps\harmdieps,\qquad
\eps^{-1}\cdcznom
=\harmdieps\oplus_\eps\eps^{-1}\overline{\rot\rcnom},$$
and hence by \eqref{hzgrad} or \eqref{hzrot} the refined decomposition
\begin{align*}
\Ltom = \overline{\na\Hoctom} \oplus_\eps \harmdieps \oplus_\eps \eps^{-1}\overline{\rot\rcnom}
\end{align*}
follows. Again from \eqref{hzgrad} and \eqref{hzrot} we obtain
\begin{align*}
\rctom
&=\rcztom\oplus_\eps\big(\rctom\cap\eps^{-1}\overline{\rot\rcnom}\big),\\
\eps^{-1}\dcnom
&=\big(\eps^{-1}\dcnom\cap\overline{\na\Hoctom}\big)\oplus_\eps\eps^{-1}\dcznom,
\end{align*}
and thus the further refinements
\begin{align}
\label{rotRHZ}
\rot\rctom
&=\rot\big(\rctom\cap\eps^{-1}\overline{\rot\rcnom}\big)
=\rot\big(\rctom\cap\eps^{-1}\dcznom\cap\harmdieps^{\perp_\eps}\big),\\
\label{divDHZ}
\div\dcnom
&=\div\big(\dcnom\cap\eps\,\overline{\na\Hoctom}\big)
=\div\Big(\dcnom\cap\eps\,\big(\rcztom\cap\harmdieps^{\perp_\eps}\big)\Big).
\end{align}
hold. With the help of these representations the closedness 
of $\rot\rctom$ and $\div\dcnom$ follows immediately by Theorem \ref{MA}. 
The closedness of $\na\Hoctom$ follows with the standard Poincar\'e-Friedrichs inequality
from Theorem \ref{rellich} (ii).

\begin{lem}
\label{bildabg}
The range spaces $\na\Hoctom$, $\rot\rctom$ and $\div\dcnom$ are closed subspaces of $\Ltom$.
\end{lem} 

\begin{proof}
Suppose $(H_{n})\subset\rot\rctom$ with $H_{n}\to H$ in $\Ltom$.
By the representation \eqref{rotRHZ} there exists a sequence 
$(E_{n})\subset\rctom\cap\eps^{-1}\dcznom\cap\harmdieps^{\perp_\eps}$
with $\rot E_{n}=H_{n}$. Theorem \ref{MA} yields
$$\normsLtom{E_n}\leq c_{\mathsf{m}}\normsLtom{H_n},$$
i.e., $(E_n)$ is a Cauchy sequence in $\rctom\cap\eps^{-1}\dcznom\cap\harmdieps^{\perp_\eps}$
and hence converges in $\rom$ to some $E\in\rctom\cap\eps^{-1}\dcznom\cap\harmdieps^{\perp_\eps}$. 
Thus $H\leftarrow\rot E_{n}\to\rot E\in\rot\big(\rctom\cap\eps^{-1}\dcznom\cap\harmdieps^{\perp_\eps}\big)$.
Analogously we show that $\div\dcnom$ is closed.
\end{proof}

Altogether we obtain

\begin{theo}
\label{HZ}
The following orthogonal decompositions hold:
\begin{align*}
\Ltom
&=\na\Hoctom\oplus_\eps \eps^{-1}\dcznom
=\rcztom\oplus_\eps\eps^{-1}\rot\rcnom\\
&=\na\Hoctom\oplus_\eps\harmdieps\oplus_\eps\eps^{-1}\rot\rcnom.
\end{align*}
Furthermore
\begin{align*}
\rot\rctom
&=\rot\big(\rctom\cap\eps^{-1}\rot\rcnom\big)
=\rot\big(\rctom\cap\eps^{-1}\dcznom\cap\harmdieps^{\perp_\eps}\big),\\
\div\dcnom
&=\div\big(\dcnom\cap\eps\na\Hoctom\big)
=\div\Big(\dcnom\cap\eps\big(\rcztom\cap\harmdieps^{\perp_\eps}\big)\Big)
\end{align*}
and
\begin{align*}
\na\Hoctom
&=\rcztom\cap\harmdieps^{\perp_\eps},&
\div\dcnom
&=\begin{cases}
\Ltom&\text{, if }\Gamma_{\nu}\neq\Gamma,\\
\Ltperpom&\text{, if }\Gamma_{\nu}=\Gamma,
\end{cases}\\
\rot\rcnom
&=\dcznom\cap\harmdieps^{\perp}.
\end{align*}
Moreover, the scalar $\na$- and vector $\rot$-, $\div$-potentials are uniquely determined
in $\Hoctom$ \big(or in $\Hoperpom$, if $\Gamma_\tau=\emptyset$\big),
$\rctom\cap\eps^{-1}\dcznom\cap\harmdieps^{\perp_\eps}$
and $\dcnom\cap\eps\big(\rcztom\cap\harmdieps^{\perp_\eps}\big)$, respectively,
and depend continuously on their respective images by the Poincar\'e-Friedrichs estimate,
see Theorem \ref{rellich} (ii), and Theorem \ref{MA}.
\end{theo}

\begin{rem}
\label{HZrem}
Under more restrictive assumptions on the weak Lipschitz pair $(\om,\Gamma_{\tau})$
there exists $\Hoom$-potentials as well, 
see Remarks \ref{gradphirot0rem}, \ref{remrotpot} and \ref{remdivpot}.
More precisely, let $\om\subset\rt$ be a bounded domain:
\begin{itemize}
\item[\bf(i)] 
If $\om$ is simply connected and
$(\om,\Gamma_{\tau})$ a weak Lipschitz pair, such that $\Gamma_{\tau}$ is connected,
then, since the Poincar\'e-Friedrichs estimate holds by Theorem \ref{rellich} (ii), 
a linear and continuous potential operator
$\mathcal{S}_{\na}:\rcztom\rightarrow\Hoctom$ exists
with $\na\mathcal{S}_{\na}E=E$ for all $E\in\rcztom$. Furthermore
$$\rcztom=\na\Hoctom,\qquad
\harmdieps=\{0\}.$$
Again, if $\Gamma_\tau=\emptyset$, we have to replace $\Hoctom$ by $\Hoperpom$.
Moreover, a linear and continuous potential operator
$\mathcal{S}_{\na}:\rcztom\rightarrow\Hoctom\cap\Ho(\rt)$
can be chosen, if $\om$ is even strong Lipschitz.
\item[\bf(ii)]
If $\om$ is a strong Lipschitz domain, such that $\rt\setminus\ol\om$ is connected (i.e. $\Gamma$ is connected),
and if $\Gamma_\nu=\bigcup_{k=1}^{K}\Gamma_{\nu,k}$, $K\in\nz$, 
with disjoint, relatively open and simply connected strong Lipschitz surface patches $\Gamma_{\nu,k}\subset\Gamma$
satisfying $\dist (\Gamma_{\nu,k},\Gamma_{\nu,\ell})>0$ for all $1\leq k\neq\ell\leq K$,
then there exists a linear and continuous potential operator
$\mathcal{S}_{\mathsf{r}}:\dcznom\rightarrow\Hocnom\cap\Ho(\rt)$
with $\rot\mathcal{S}_{\mathsf{r}}H=H$ for all $H\in\dcznom$. Furthermore
$$\dcznom=\rot\Hocnom,\qquad
\harmdieps=\{0\}.$$
\item[\bf(iii)]
If $\om$ is a strong Lipschitz domain with 
$\Gamma_\nu=\overset{K}{\underset{k=1}{\dot\bigcup}}\Gamma_{\nu,k}$, $K\in\nz$,
where $\dist(\Gamma_{\nu,k},\Gamma_{\nu,\ell})>0$ for all $1\leq k\neq\ell\leq K$
and $\Gamma_{\nu,k}$ are $\Cont^{3}$-boundary patches allowing 
for $\Cont^{3}$-regular extensions $\hat\om_{\nu,k}$
having connected complements $\rt\setminus\ol{\hat\om}_{\nu,k}$,
then there exists a linear and continuous potential operator
$\mathcal{S}_{\mathsf{d}}:\Ltom\rightarrow\Hocnom\cap\Ho(\rt)$
with $\div\mathcal{S}_{\mathsf{d}}u=u$ for all $u\in\Ltom$. Furthermore
$$\Ltom=\div\Hocnom.$$
If $\Gamma_\nu=\Gamma$, we have to replace $\Ltom$ by $\Ltperpom$.
\item[\bf(iii')] 
As noted earlier at the end of Section \ref{secpotwithbc}, the results of (iii) are not optimal.
Using different techniques from \cite[Lemma 2.1.1]{sohrbook}, it has been shown 
in \cite[Lemma 3.2]{bauerneffpaulystarkedevDivsymCurl}, that the assertions of (iii)
hold for any strong Lipschitz pair $(\om,\Gamma_{\nu})$.
\end{itemize}
\end{rem}

We will point out how the potentials in the latter Helmholtz decompositions can be computed.
By Theorem \ref{HZ} any vector field $E\in\Ltom$ can be written as
\begin{align*}
E = E_\na + E_{\mathcal{H}} + \eps^{-1}E_{\mathsf{r}} 
\in \na\Hoctom \oplus_\eps \harmdieps \oplus_\eps\eps^{-1}\rot\rcnom.
\end{align*}
Interchanging the roles of $\Gamma_{\nu}$ und $\Gamma_{\tau}$ 
in the decompositions of Theorem \ref{HZ} yields (with $\eps = \id$)
\begin{align}
\label{rotrangewithtilde}
\rot\rcnom 
=\rot\mathsf{X}(\om),\qquad
\mathsf{X}(\om)
:=\rcnom\cap\rot\rctom
=\rcnom\cap\dcztom\cap\tilde{\mathcal{H}}(\om)^\perp,
\end{align}
where $\tilde{\mathcal{H}}(\om) = \rcznom \cap\dcztom$, 
which is also finite dimensional.
Hence $E_\na = \na u$ and $E_{\mathsf{r}}=\rot H$ with uniquely determined
$$u\in\Hoctom,\qquad
H\in\mathsf{X}(\om),$$
i.e., $E$ can be written as
\begin{align}
\label{HZE}
E = \na u + E_{\mathcal{H}} + \eps^{-1}\rot H.
\end{align}
In order to calculate $u$ we test \eqref{HZE} with 
$\na \varphi$, $\varphi\in\Hoctom$, and due to orthogonality we get
\begin{align}
\label{varu}
\forall\,\varphi\in\Hoctom\qquad
\scpLtom{\eps\na u}{\na \varphi} 
= \scpLtom{\eps E}{\na\varphi}.
\end{align}
In case $\Gamma_{\tau}=\emptyset$ we set again $\Hoctom:=\Hoperpom$.
By Poincar\'e's estimate the Lax-Milgram lemma or simply Riesz' representation theorem 
yields a unique $u\in\Hoctom$ satisfying \eqref{varu} and 
$$\norms{u}_{\Hoom} \leq c\,\normsLtom{E}.$$
To calculate $H\in \mathsf{X}(\om)$ 
we test \eqref{HZE} with $\rot\Phi$, $\Phi\in \mathsf{X}(\om)$, 
and again get by orthogonality
\begin{align}
\label{varA}
\forall\,\Phi\in \mathsf{X}(\om)\qquad
\scpLtom{\eps^{-1}\rot H}{\rot\Phi} 
= \scpLtom{E}{\rot\Phi}.
\end{align}
Due to the Maxwell estimate, i.e., Theorem \ref{MA}, 
$\scpLtom{\eps^{-1}\rot\,\cdot\,}{\rot\,\cdot\,}$ is a coercive bilinear form
or even a scalar product on $\mathsf{X}(\om)$ 
and Lax-Milgram's lemma or Riesz' representation theorem 
yields a unique $H\in \mathsf{X}(\om)$, 
satisfying \eqref{varA} and
$$\norms{H}_{\rom}\leq c\,\normsLtom{E}.$$
\eqref{varu} shows by definition and Theorem \ref{theoweakeqstrong} 
$$\eps(E-\na u)\in\cdcznom=\dcznom.$$
Furthermore \eqref{varA} holds for all $\Phi\in\rcnom$ by \eqref{rotrangewithtilde} as well.
Hence by definition and Theorem \ref{theoweakeqstrong} we obtain
$$E-\eps^{-1}\rot H\in\crcztom=\rcztom.$$
Finally the Dirichlet-Neumann field $E_{\mathcal{H}}$ is given by
$$E_{\mathcal{H}}:=E-E_\na-\eps^{-1}E_{\mathsf{r}}=E-\na u-\eps^{-1}\rot H
\in\harmdieps$$
as
$$E-\na u,\eps^{-1}\rot H\in\eps^{-1}\dcznom,\qquad
E-\eps^{-1}\rot H,\na u\in\rcztom.$$

\begin{rem}
\eqref{varu} is the variational formulation of the classical boundary value problem
\begin{align}
\div\eps\na u&=\div\eps E&&\text{in }\om,\label{potu1}\\
u&=0&&\text{on }\Gamma_{\tau},\non\\
n\cdot\eps\na u&=n\cdot\eps E&&\text{on }\Gamma_{\nu}\label{potu3},
\end{align}
because in the smooth case \eqref{varu} yields for all $\varphi\in\Hoctom$
$$0=\scpLtom{\eps(\na u-E)}{\na\varphi}
=-\underbrace{\scpLtom{\div\eps(\na u-E)}{\varphi}}_{=0}
+\int_{\Gamma_{\nu}}(n\cdot\eps(\na u- E))\varphi.$$
Here, \eqref{potu1} has a proper meaning not only in $\Sobolev^{-1}(\om)=\big(\hocom\big)'$
but also in $\big(\Hoctom\big)'$, 
which gives meaning to the Neumann boundary condition \eqref{potu3} in this dual space as well.
\eqref{varA} is the variational formulation of the classical boundary value problem
\begin{align}
\rot\eps^{-1}\rot H&=\rot E&&\text{in }\om,\label{potH1}\\
\div H&=0&&\text{in }\om,\non\\
n\times H&=0&&\text{on }\Gamma_{\nu},\non\\
n\cdot H&=0&&\text{on }\Gamma_{\tau},\non\\
n\times\eps^{-1}\rot H&=n\times E&&\text{on }\Gamma_{\tau},\label{potH5}\\
H&\perp\tilde{\mathcal{H}}(\om),\non
\end{align}
because in the smooth case \eqref{varA} yields for all $\Phi\in\rcnom$
$$0=\scpLtom{\eps^{-1}(\rot H - \eps E)}{\rot\Phi}
= \underbrace{\scpLtom{\rot\eps^{-1}(\rot H - \eps E)}{\Phi}}_{=0}
+\int_{\Gamma_\tau} (n\times \eps^{-1}(\rot H -\eps E)\cdot \Phi.$$
\eqref{potH1} has proper meaning not only in $\Sobolev^{-1}(\om)=\big(\hocom\big)'$
but also in $\big(\Hocnom\big)'$ or even in $\big(\rcnom\big)'$, 
which gives meaning to the Neumann boundary condition \eqref{potH5} in this dual space as well.
\end{rem}

\subsection{Static solution theory}

As a further application we turn to the boundary value problem 
of electro- and magnetostatics with mixed boundary values: 
Let $F\in\Ltom$, $g\in\Ltom$, 
$E_\tau\in\rotspace(\om)$, $\eps^{-1}E_\nu\in\divspace(\om)$ and $\eps$ be admissible. 
Find $E\in\rom\cap\eps^{-1}\dom$ with
\begin{align}
\label{EMSRWProt}
\rot E &= F,\\
\label{EMSRWPdiv}
\div \eps E &= g,\\
\label{EMSRWPtang}
E - E_{\tau} &\in\rctom,\\
\label{EMSRWPnorm}
\eps(E - E_{\nu}) &\in\dcnom.
\end{align}
For uniqueness, we require the additional conditions
\begin{align}
\label{ONBproj}
\scpLtom{\eps E}{D_\ell} = \alpha_\ell \in\rz,\quad \ell=1,\dots,d,
\end{align}
where $d$ is the dimension and $\{D_\ell\}$ an $\eps$-orthonormal basis of $\harmdieps$.
The boundary values on $\Gamma_{\tau}$ and $\Gamma_{\nu},$ respectively, 
are realized by the given fields $E_{\tau}$ and $E_{\nu}$, respectively.

Let us solve the problem: Theorem \ref{HZ} yields
\begin{align*}
\rom&=\rcztom\oplus_\eps\big(\rom\cap\eps^{-1}\rot\rcnom\big)
=\rcztom\oplus_\eps\big(\rom\cap\eps^{-1}\dcznom\cap\harmdieps^{\perp_\eps}\big).
\end{align*}
The decomposition $E_{\tau} = E_{\tau,\mathsf{r}} + E_{\tau,\mathsf{d}}$ 
with $ E_{\tau,\mathsf{r}} \in \rcztom $ 
and $ E_{\tau,\mathsf{d}} \in \rom\cap\eps^{-1}\dcznom\cap\harmdieps^{\perp_\eps}$ shows
$$E-E_{\tau,\mathsf{d}} 
= E - E_{\tau} + E_{\tau} - E_{\tau,\mathsf{d}} 
= E - E_{\tau} + E_{\tau,\mathsf{r}}\in\rctom.$$
Hence
$$E-E_{\tau}\in\rctom 
\qequi 
E-E_{\tau,\mathsf{d}}\in\rctom,$$
which means the field $E_{\tau}$ realizing the boundary values on $\Gamma_{\tau}$ 
can w.l.o.g.~be chosen from the more regular space $\rom\cap\eps^{-1}\dcznom\cap\harmdieps^{\perp_\eps}$.
Similarly
\begin{align*}
\eps^{-1}\dom
=\eps^{-1}\dcznom\oplus_\eps\big(\eps^{-1}\dom\cap\na\Hoctom\big)
=\eps^{-1}\dcznom\oplus_\eps\big(\eps^{-1}\dom\cap\rcztom\cap\harmdieps^{\perp_\eps}\big)
\end{align*}
shows for $E_{\nu} = E_{\nu,\mathsf{d}} + E_{\nu,\mathsf{r}}$
with $E_{\nu,\mathsf{d}} \in \eps^{-1}\dcznom $
and $E_{\nu,\mathsf{r}} \in \eps^{-1}\dom\cap\rcztom\cap\harmdieps^{\perp_\eps} $ 
$$E-E_{\nu,\mathsf{r}} 
= E - E_{\nu} + E_{\nu} - E_{\nu,\mathsf{r}} 
= E - E_{\nu} + E_{\nu,\mathsf{d}} \in\eps^{-1}\dcnom.$$
Hence
$$E-E_{\nu}\in\dcnom 
\qequi 
E-E_{\nu,\mathsf{r}}\in\eps^{-1}\dcnom,$$
so the field $E_{\nu}$ can w.l.o.g. be chosen from $\eps^{-1}\dom\cap\rcztom\cap\harmdieps^{\perp_\eps}$.
Therefore we will work with the boundary conditions 
$$E_{\tau,\mathsf{d}}\in\rom\cap\eps^{-1}\dcznom\cap\harmdieps^{\perp_\eps},\qquad
E_{\nu,\mathsf{r}}\in\eps^{-1}\dom\cap\rcztom\cap\harmdieps^{\perp_\eps}$$
and observe
$$\norms{E_{\tau,\mathsf{d}}}_{\Ltepsom}\leq\norms{E_{\tau}}_{\Ltepsom},\quad
\rot E_{\tau,\mathsf{d}}=\rot E_{\tau},\qquad
\norms{E_{\nu,\mathsf{r}}}_{\Ltepsom}\leq\norms{E_{\nu}}_{\Ltepsom},\quad
\div\eps E_{\nu,\mathsf{r}}=\div\eps E_{\nu}.$$
We first note that the system admits at most one solution, 
as for the homogeneous problem $E\in\harmdieps$ together with \eqref{ONBproj} yield $E=0$. 
Turning to existence the conditions
$$\rot (E-E_{\tau,\mathsf{d}})=F-\rot E_{\tau,\mathsf{d}}=:\tilde F \in\rot\rctom,\quad
\div \eps (E-E_{\nu,\mathsf{r}}) = g-\div\eps E_{\nu,\mathsf{r}}=:\tilde g \in\div\dcnom$$
are necessary and from now on shall be assumed. By setting
$$\tilde{E} := E - E_{\tau,\mathsf{d}} - E_{\nu,\mathsf{r}}$$
the problem is transformed into a problem with homogenous boundary conditions, 
i.e., $\tilde E$ must solve
\begin{align}
\label{EMSRWProthomrand}
\rot \tilde E &= \tilde F,\\
\label{EMSRWPdivhomrand}
\div \eps\tilde E &= \tilde g,\\
\label{EMSRWPhomrand}
\tilde E &\in \rctom\cap\eps^{-1}\dcnom.
\end{align}
Necessary conditions for the existence of solutions are 
$$\tilde F \in \rot\rctom, \quad \tilde g \in \div \dcnom,$$
which have been assumed above. The conditions are already sufficient, 
as Theorem \ref{HZ} shows the existence of fields 
$$\tilde{E}_{\mathsf{r}}
\in\rctom\cap\eps^{-1}\dcznom\cap\harmdieps^{\perp_\eps},\qquad
\eps\tilde{E}_{\mathsf{d}}
\in\dcnom\cap\eps\big(\rcztom\cap\harmdieps^{\perp_\eps}\big)$$
with $\rot \tilde E_{\mathsf{r}} = \tilde F$ and $\div\eps\tilde E_{\mathsf{d}} = \tilde g$. Then
$$\tilde E:=\tilde E_{\mathsf{r}} + \tilde E_{\mathsf{d}} 
\in\rctom\cap\eps^{-1}\dcnom\cap\harmdieps^{\perp_\eps}$$
solves the system \eqref{EMSRWProthomrand}-\eqref{EMSRWPhomrand} 
with $\tilde E \perp_\eps \harmdieps$. Therefore 
$$E_0:=\tilde E+ E_{\tau,\mathsf{d}} + E_{\nu,\mathsf{r}}
\in\rom\cap\eps^{-1}\dom\cap\harmdieps^{\perp_\eps}$$ 
solves \eqref{EMSRWProt}-\eqref{EMSRWPnorm} 
with $E_0 \perp_\eps\harmdieps$, i.e., \eqref{ONBproj} with $\alpha = 0$. Finally 
$$E:=E_0 + \sum_{\ell} \alpha_\ell D_\ell\in\rom\cap\eps^{-1}\dom$$ 
solves \eqref{EMSRWProt}-\eqref{ONBproj}. 
Furthermore, $ E $ depends continuously on the data by Theorem \ref{MA}. 
More precisely, equipping $\rom$ and $\eps^{-1}\dom$ with the norms
$$\norms{\,\cdot\,}_{\rom}^2:=\norms{\,\cdot\,}_{\Ltepsom}^2+\normsLtom{\rot\,\cdot\,}^2,\quad
\norms{\,\cdot\,}_{\eps^{-1}\dom}^2:=\norms{E\,\cdot\,}_{\Ltepsom}^2+\normsLtom{\div\eps\,\cdot\,}^2,$$
we have
\begin{align*}
\norms{E}_{\rom\cap\eps^{-1}\dom}^2
&=\norms{E}_{\Ltepsom}^2
+\normsLtom{F}^2
+\normsLtom{g}^2,\\
\norms{E}_{\Ltepsom}^2
&=\norms{E_{0}}_{\Ltepsom}^2
+|\alpha|^2,\\
\normsLtom{E_{0}}
&\leq\norms{\tilde E}_{\Ltepsom}
+\norms{E_{\tau,\mathsf{d}}+E_{\nu,\mathsf{r}}}_{\Ltepsom},\\
\norms{E_{\tau,\mathsf{d}}+E_{\nu,\mathsf{r}}}_{\Ltepsom}^2
&=\norms{E_{\tau,\mathsf{d}}}_{\Ltepsom}^2
+\norms{E_{\nu,\mathsf{r}}}_{\Ltepsom}^2
\leq\norms{E_{\tau}}_{\Ltepsom}^2
+\norms{E_{\nu}}_{\Ltepsom}^2
\end{align*}
and by the Maxwell estimate Theorem \ref{MA}
$$\norms{\tilde E}_{\Ltepsom}
\leq c_{\mathsf{m}}\,\big(\normsLtom{\tilde F}
+\normsLtom{\tilde g}\big)
\leq c_{\mathsf{m}}\,\big(\normsLtom{F}
+\normsLtom{g}
+\normsLtom{\rot E_{\tau}}
+\normsLtom{\div\eps E_{\nu}}\big).$$
Thus there exists a constant $\tilde c_{\mathsf{m}}>0$,
just depending on $c_{\mathsf{m}}$, such that
\begin{align*}
\norms{E}_{\rom\cap\eps^{-1}\dom}
&\leq \tilde c_{\mathsf{m}}\,\big(\normsLtom{F}
+\normsLtom{g}
+\norms{E_{\tau}}_{\rom}
+\norms{E_{\nu}}_{\eps^{-1}\dom}
+|\alpha|\big).
\end{align*}
Our latter results show that the linear solution operator with 
$$(F,g,E_{\tau},E_{\nu},\alpha)\mapsto E$$
maps continuously onto $\rom\cap\eps^{-1}\dom$.
Let us summarize:

\begin{theo}
\label{satzloesung}
\eqref{EMSRWProt}-\eqref{EMSRWPnorm} admits a solution, if and only if 
$$E_{\tau}\in\rom,\quad 
E_{\nu}\in\eps^{-1}\dom,\quad 
F-\rot E_{\tau}\in\rot\rctom,\quad 
g-\div\eps E_{\nu}\in\div\dcnom.$$ 
The solution $E\in\rom\cap\eps^{-1}\dom$ can be chosen in a way, 
such that condition \eqref{ONBproj} with $\alpha\in\rz^d$ is satisfied, 
which then uniquely determines the solution. 
Furthermore the solution depends linearly and continuously on the data.
\end{theo}

We note that with Theorem \ref{HZ} the ranges can be described by
$$\rot\rcnom
=\dcznom\cap\harmdieps^{\perp},\qquad
\div\dcnom
=\begin{cases}
\Ltom&\text{, if }\Gamma_{\nu}\neq\Gamma,\\
\Ltperpom&\text{, if }\Gamma_{\nu}=\Gamma.
\end{cases}$$

For homogeneous boundary data, i.e., $E_{\tau}=E_{\nu}=0$, 
we can state a sharper result: The linear static Maxwell-operator
$$\Abb{M}{\rctom\cap\eps^{-1}\dcnom}{\rot\rctom\times\div\dcnom\times\rz^d}
{E}{\big(\rot E,\div\eps E,(\scpLtom{\eps E}{D_\ell})_{\ell=1}^{d}\big)}$$
is a topological isomorphism. Its inverse $M^{-1}$ maps not only continuously
onto $\rctom\cap\eps^{-1}\dcnom$, but also compactly into $\Ltom$ by Theorem \ref{satzMKE}.
For homogeneous kernel data, i.e., for 
$$\Abb{M_{0}}{\rctom\cap\eps^{-1}\dcnom\cap\harmdieps^{\perp_\eps}}{\rot\rctom\times\div\dcnom}
{E}{(\rot E,\div\eps E)}$$
we have $\norms{M_{0}^{-1}}\leq(c_{\mathsf{m}}^2+1)^{1/2}$.

\appendix

\section{Proof of Lemma \ref{Hoct}}

We will show the density result for $ \Ho $-functions, Lemma \ref{Hoct},
which is first proved for a flat boundary in Lemma \ref{lemmahalfspace} 
and then generalized to weak Lipschitz pairs in Lemma \ref{Hoct}. 
Although both proofs can be found in \cite[Lemma 2,3]{jochmanncompembmaxmixbc},
we repeat them here, using our notation and with some major simplifications.
Let us introduce the following notations:
$$\rz_{\pm}^3:=\{x\in\rt:\pm x_3>0\},\qquad
\rt_0:=\{x\in\rt:x_3=0\},\qquad
\rt_{0,-}:=\{x\in\rt_0:x_1<0\}$$

\begin{lem}
\label{lemmahalfspace}
Suppose $u_-\in\Ho(\rt_-)$ with compact support and $u_-|_{\rt_{0,-}}=0$ in the sense of traces. Then 
$$\forall\,\delta>0\quad
\exists\,u\in\Cic(\rt\setminus \rt_{0,-})\qquad
\norms{u-u_-}_{\Ho(\rt_-)}<\delta.$$
\end{lem}

\begin{proof}
Let $\delta>0$. For $t>0$ define $u_-^t(x):=u_-(x_1 - t, x_2, x_3)$ f.a.a. $ x \in \rt_- $,
the translation of $u_-$ in the direction of $x_1$. 
Then $u_-^t\in\Ho(\rt_-)$ and its trace satisfies
$$u_-^t(x)=0\text{ f.a.a. }
x\in\rt_{0}\text{ with }x_1\leq t.$$
Since the translation is continuous, there exists $t>0$ with
$\norms{u_-^t-u_-}_{\Ho(\rt_-)}<\delta/3$.
The reflection $u_+^t$ of $u_-^t$
defined by $u_+^t(x):=u_-^t(x_{1},x_{2},-x_3)$ for $x\in\rt_+$
defines an element of $\Ho(\rt_+)$ with
$u_+^t(x)=u_-^t(x)$ f.a.a. $x\in\rt_0$ and hence
$$u_+^t(x)=u_-^t(x)=0
\text{ f.a.a. }
x\in\rt_0\text{ with }x_1 \leq t.$$
A cut-off argument\footnote{Multiply $u_+^t(x)$
with some $\varphi(x_{1})$, where $\varphi\in\Ci(\rz,[0,1])$
with $\varphi|_{(-\infty,t/2)}=0$ and $\varphi|_{(t,\infty)}=1$.} 
yields a function $\tilde u_+^t\in\Ho(\rt_+)$ with
$$\tilde u_+^t(x)=0\text{ f.a.a. }x\in\rt_+
\text{ with }x_1\leq t/2,\qquad
\tilde u_+^t(x)=u_+^t(x)=u_-^{t}(x)\text{ f.a.a. }x\in\rt_0.$$
Hence
\begin{align*}
u^t:=
\begin{cases}
\tilde u_+^t&\text{in }\rt_+,\\
u_-^t&\text{in }\rt_-
\end{cases}
\end{align*}
defines an element in $\Ho(\rt)$, which vanishes f.a.a. $x\in\rt$
with $x_{3}\geq0$ and $x_{1}\leq t/2$. 
Let $u^{t,s}$ with $u^{t,s}(x):=u^t(x_{1},x_{2},x_3+s)$ for $x\in\rt$ 
be the translation of $u^t$ in the direction $-x_3$. 
Then $u^{t,s}\in\Ho(\rt)$ and there is $s>0$ with
$\norms{u^{t,s}-u^t}_{\Ho(\rt)}<\delta/3$.
Furthermore $u^{t,s}=0$ in a neighbourhood of $\rt_{0,-}$ even of
$\{x\in\rt_{+}\,:\,x_{1}<t/4\}$, more explicitly in
$$\big\{x\in\rt: x_3>-s~\text{and}~x_1 < \frac{t}{2}\big\} \supset \rt_{0,-}.$$
Since $\supp u^{t,s}\Subset\rt$, Friedrichs' mollification yields some 
$u\in\Cic(\rt\setminus \rt_{0,-})$ with $\norms{u-u^{t,s}}_{\Ho(\rt)} < \delta /3$.
Finally
\begin{align*}
\norms{u-u_-}_{\Ho(\rt_-)}
&\leq\norms{u-u^{t,s}}_{\Ho(\rt_-)}
+\norms{u^{t,s}-u^{t}}_{\Ho(\rt_-)}+\norms{u^{t}-u_-}_{\Ho(\rt_-)}\\
&\leq\norms{u-u^{t,s}}_{\Ho(\rt)}
+\norms{u^{t,s}-u^{t}}_{\Ho(\rt)}+\norms{u^{t}_{-}-u_-}_{\Ho(\rt_-)},
\end{align*}
completing the proof.
\end{proof}

\begin{proofof}{Lemma \ref{Hoct}}
We already know $\Hoctom\subset\cHoctom$. Any $u\in\cHoctom$ can be extended by zero into
an open neighbourhood of $\Gamma_{\tau}$, 
where the common boundary of $\om$ and the extension is exactly $\Gamma_{\tau}$, and stays in $\Ho$. 
For this, let us denote the extended domain by $\hat\om$ and the extended function by $\hat u$.
Then $\hat u$ belongs to $\Ho(\hat\om)$,
since for all $\varphi\in\Cic(\hat\om)$ it follows $\varphi\in\Cicnom$ and hence
$$\scp{\hat u}{\p_{i}\varphi}_{\Lt(\hat\om)}
=\scpLtom{u}{\p_{i}\varphi}
=-\scpLtom{\p_{i}u}{\varphi}
=-\scp{\widehat{\p_{i}u}}{\varphi}_{\Lt(\hat\om)}.$$
Since $\hat u$ is zero in the extension part of $\hat\om$, its trace vanishes on $\Gamma_{\tau}$
and coincides with the trace of $u$ itself, which shows $u\in\Hocttr(\om)$. Therefore, we have
$$\Hoctom\subset\cHoctom\subset\Hocttr(\om).$$
The reverse inclusion is proved with the help of Lemma \ref{lemmahalfspace}. 
Suppose $u\in\Hocttr(\om)$. Let $U_1,\dots,U_K\subset\rt$ 
be an open covering of $\Gamma$, see Definitions \ref{defilipmani} and \ref{defilipsubmani},
and set $U_0=\om$. Moreover, let $\chi_k$, $k=0,\dots,K$, denote a partition of unity 
subordinate to the open covering $U_0,\dots,U_K$ of $\ol\om$. 
Since $\supp\chi_{0}\Subset\om$, we obtain by mollification 
$u_0:=\chi_0u\in\Hocom\subset\Hoctom$.
Let $k=1,\dots,K$. Then it is sufficient to show that
$$u_{k}:=\chi_ku\in\Hocttr(\om)$$
belongs to $\Hoctom$, since then $u=\sum_{k=0}^{K}u_{k}\in\Hoctom$.
Hence, by Remark \ref{remlipclosure} (i), it is sufficient to show that $u_{k}$, $k=1,\dots,K$,
can be approximated in $\hoom$ by $\Czoct(\om)$ functions. 
We will utilize the notations of Subsection \ref{mcpweaklip}. Then
$$u_{k}\in\Hgentr{1}{{\hat\Gamma_{\tau,k}}}{\circ}(\om_{k}),\qquad
\tilde u_k:=u_k\circ \psi_k\in\Hgentr{1}{\hat\gamma_{\tau}}{\circ}(\Xi),$$
and there are three cases, i.e., $\gamma_{\tau}\in\{\emptyset,B_{0},B_{0,-}\}$, to be discussed.
Let $k=1,\dots,K$ and $\delta>0$.
\begin{description}
\item[$\gamma_{\tau}=B_{0}$] 
Then $\tilde u_k\in\Hgentr{1}{\gamma}{\circ}(\Xi)=\Hoctr(\Xi)=\Hoc(\Xi)$ 
by classical results. Hence there exists $\tilde\varphi\in\Cic(\Xi)$ 
with $\norms{\tilde\varphi-\tilde u_k}_{\Ho(\Xi)}<\delta$. 
Hence $\varphi:=\tilde\varphi\circ\phi_k\in\Czoc(\om_{k})\subset\Czoc(\om)\subset\Czoct(\om)$ 
and by Lemma \ref{lemtrafo} we obtain
\begin{align}
&\qquad\norms{\varphi-u_k}_{\hoom}^2
=\norms{\varphi-u_k}_{\Ho(\om_{k})}^2
=\int_{\om_k}\big(|\varphi - u_k|^2
+|\na\varphi-\na u_k|^2\big)\non\\
&=\int_{\Xi}|\det \psi_k'|\big(|\tilde\varphi-\tilde u_k|^2
+|\widetilde{\na\varphi}-\widetilde{\na u_k}|^2\big)
=\int_{\Xi}|\det\psi_k'|\Big(|\tilde\varphi-\tilde u_k|^2
+\big|(\psi_k')^{-\top}(\na\tilde\varphi-\na\tilde u_k)\big|^2\Big)\non\\
&\leq c\int_{\Xi}\big(|\tilde\varphi-\tilde u_k|^2
+|\na\tilde\varphi-\na\tilde u_k|^2\big)
=c\,\norms{\tilde\varphi - \tilde u_k}_{\Ho(\Xi)}^2
<c\,\delta^2.
\label{wstrafo}
\end{align}
\item[$\gamma_{\tau}=\emptyset$]
Then $\tilde u_k\in\Hgentr{1}{\sigma}{\circ}(\Xi)\subset\Ho(\Xi)$ 
and by Calderon's extension theorem $\tilde u_k$ can be identified with $\tilde u_k\in\Ho(\rt)$.
Thus there exists $\tilde\varphi\in\Cic(\rt)$ 
with $\norms{\tilde\varphi-\tilde u_k}_{\Ho(\rt)}<\delta$
and $\varphi:=\tilde\varphi\circ\phi_k \in\Czo(\om_{k})$.
Let $\eta_k\in\Cic(U_k)$ with $\eta_k=1$ on $\supp\chi_k$.
We set $\varphi_k:=\eta_k\varphi\in\Czoc_{\Sigma_{k}}(\om_{k})\subset\Czoct(\om)$
and note that $u_k =\eta_k u_k$ does not change.
With $\varphi_{k}-u_k=\eta_k(\varphi-u_{k})$ and \eqref{wstrafo} we get
$$\norms{\varphi_{k}-u_k}_{\hoom}
\leq\norms{\varphi-u_k}_{\Ho(\om_{k})}
\leq c\,\norms{\tilde\varphi - \tilde u_k}_{\Ho(\Xi)}
<c\,\delta.$$
\item[$\gamma_{\tau}=B_{0,-}$] 
Then $\tilde u_k\in\Hgentr{1}{\hat\gamma_{\tau}}{\circ}(\Xi)$ and we identify
$\tilde u_k$ with its extension by zero to $\rt_{-}$.
Therefore $\tilde u_k\in\Ho(\rt_-)$ has compact support in $\ol\Xi$ 
and satisfies $\tilde u_k|_{\rt_{0,-}}=0$. 
Then Lemma \ref{lemmahalfspace} yields some $\tilde\varphi\in\Cic(\rt\setminus \rt_{0,-})$ 
with $\norms{\tilde\varphi-\tilde u_k}_{\Ho(\rt_-)}<\delta$.
Hence $\tilde\varphi\in\Cic_{\gamma_{\tau}}(\Xi)$
and thus $\varphi:=\tilde\varphi\circ\phi_k\in\Czoc_{\Gamma_{\tau,k}}(\om_k)$.
Again, let $\eta_k\in\Cic(U_k)$ with $\eta_k=1$ on $\supp\chi_k$
and define $\varphi_k:=\eta_k\varphi\in\Czoc_{\hat\Gamma_{\tau,k}}(\om_k)\subset\Czoct(\om)$.
Noting $u_k =\eta_k u_k$ does not change 
and with $\varphi_{k}-u_k=\eta_k(\varphi-u_{k})$ and \eqref{wstrafo} we see
$$\norms{\varphi_{k}-u_k}_{\hoom}
\leq\norms{\varphi-u_k}_{\Ho(\om_{k})}
\leq c\,\norms{\tilde\varphi - \tilde u_k}_{\Ho(\Xi)}
<c\,\delta.$$
\end{description}
\end{proofof}

\section*{Acknowledgements}
We express our cordial gratitude to Immanuel Anjam 
for providing the graphics for the figures of this paper.


\bibliographystyle{plain} 
\bibliography{paule}

\begin{thebibliography}{10}

\bibitem{bauerneffpaulystarkedevDivsymCurl}
S.~Bauer, P.~Neff, D.~Pauly, and G.~Starke.
\newblock {D}ev-{D}iv- and {D}ev{S}ym-{D}ev{C}url-{I}nequalities for
  {I}ncompatible {S}quare {T}ensor {F}ields with {M}ixed {B}oundary
  {C}onditions.
\newblock {\em ESAIM Control Optim. Calc. Var.}, 22:112--133, 2016.

\bibitem{bauerpaulyschomburgmaxcompweaklip}
S.~Bauer, D.~Pauly, and M.~Schomburg.
\newblock The {M}axwell compactness property in bounded weak {L}ipschitz
  domains with mixed boundary conditions.
\newblock {\em SIAM J. Math. Anal.}, 48(4):2912--2943, 2016.

\bibitem{bauerpaulyschomburgmaxcompweakliprNarxiv}
S.~Bauer, D.~Pauly, and M.~Schomburg.
\newblock {W}eck's selection theorem: The {M}axwell compactness property for
  bounded weak {L}ipschitz domains with mixed boundary conditions in arbitrary
  dimensions.
\newblock {\em arXiv, https://arxiv.org/abs/1809.01192}, 2018.

\bibitem{bauerpaulyschomburgmaxcompweakliprN}
S.~Bauer, D.~Pauly, and M.~Schomburg.
\newblock {W}eck's selection theorem: The {M}axwell compactness property for
  bounded weak {L}ipschitz domains with mixed boundary conditions in arbitrary
  dimensions.
\newblock {\em MaxwellÕs Equations: Analysis and Numerics (Radon Series on
  Computational and Applied Mathematics), De Gruyter}, 2019.

\bibitem{costabelremmaxlip}
M.~Costabel.
\newblock A remark on the regularity of solutions of {M}axwell's equations on
  {L}ipschitz domains.
\newblock {\em Math. Methods Appl. Sci.}, 12(4):365--368, 1990.

\bibitem{eiduslabp}
D.M. Eidus.
\newblock On the principle of limiting absorption.
\newblock {\em Mat. Sb. (N.S.)}, 57(99):13--44, 1962.

\bibitem{eiduslamp}
D.M. Eidus.
\newblock On the limiting amplitude principle.
\newblock {\em Dokl. Akad. Nauk SSSR}, 158:794--797, 1964.

\bibitem{eiduslamptwo}
D.M. Eidus.
\newblock The principle of limiting amplitude.
\newblock {\em Uspehi Mat. Nauk}, 24(3 (147)):91--156, 1969.

\bibitem{jochmanncompembmaxmixbc}
F.~Jochmann.
\newblock A compactness result for vector fields with divergence and curl in
  ${L}^q({\Omega})$ involving mixed boundary conditions.
\newblock {\em Appl. Anal.}, 66:189--203, 1997.

\bibitem{kuhndiss}
P.~Kuhn.
\newblock {\em Die {M}axwellgleichung mit wechselnden {R}andbedingungen}.
\newblock Dissertation, Universit\"at Essen, Fachbereich Mathematik,
  http://arxiv.org/abs/1108.2028, {\it Shaker}, 1999.

\bibitem{kuhnpaulyregmax}
P.~Kuhn and D.~Pauly.
\newblock Regularity results for generalized electro-magnetic problems.
\newblock {\em Analysis (Munich)}, 30(3):225--252, 2010.

\bibitem{leisbook}
R.~Leis.
\newblock {\em Initial Boundary Value Problems in Mathematical Physics}.
\newblock Teubner, Stuttgart, 1986.

\bibitem{paulytimeharm}
D.~Pauly.
\newblock Low frequency asymptotics for time-harmonic generalized {M}axwell
  equations in nonsmooth exterior domains.
\newblock {\em Adv. Math. Sci. Appl.}, 16(2):591--622, 2006.

\bibitem{paulystatic}
D.~Pauly.
\newblock Generalized electro-magneto statics in nonsmooth exterior domains.
\newblock {\em Analysis (Munich)}, 27(4):425--464, 2007.

\bibitem{paulyasym}
D.~Pauly.
\newblock Complete low frequency asymptotics for time-harmonic generalized
  {M}axwell equations in nonsmooth exterior domains.
\newblock {\em Asymptot. Anal.}, 60(3-4):125--184, 2008.

\bibitem{paulydeco}
D.~Pauly.
\newblock {H}odge-{H}elmholtz decompositions of weighted {S}obolev spaces in
  irregular exterior domains with inhomogeneous and anisotropic media.
\newblock {\em Math. Methods Appl. Sci.}, 31:1509--1543, 2008.

\bibitem{paulypoly}
D.~Pauly.
\newblock On polynomial and exponential decay of eigen-solutions to exterior
  boundary value problems for the generalized time-harmonic {M}axwell system.
\newblock {\em Asymptot. Anal.}, 79(1-2):133--160, 2012.

\bibitem{picardpotential}
R.~Picard.
\newblock {R}andwertaufgaben der verallgemeinerten {P}otentialtheorie.
\newblock {\em Math. Methods Appl. Sci.}, 3:218--228, 1981.

\bibitem{picardboundaryelectro}
R.~Picard.
\newblock On the boundary value problems of electro- and magnetostatics.
\newblock {\em Proc. Roy. Soc. Edinburgh Sect. A}, 92:165--174, 1982.

\bibitem{picardcomimb}
R.~Picard.
\newblock An elementary proof for a compact imbedding result in generalized
  electromagnetic theory.
\newblock {\em Math. Z.}, 187:151--164, 1984.

\bibitem{picardlowfreqmax}
R.~Picard.
\newblock On the low frequency asymptotics in electromagnetic theory.
\newblock {\em J. Reine Angew. Math.}, 354:50--73, 1984.

\bibitem{picarddeco}
R.~Picard.
\newblock Some decomposition theorems and their applications to non-linear
  potential theory and {H}odge theory.
\newblock {\em Math. Methods Appl. Sci.}, 12:35--53, 1990.

\bibitem{picardweckwitschxmas}
R.~Picard, N.~Weck, and K.-J. Witsch.
\newblock Time-harmonic {M}axwell equations in the exterior of perfectly
  conducting, irregular obstacles.
\newblock {\em Analysis (Munich)}, 21:231--263, 2001.

\bibitem{sohrbook}
H.~Sohr.
\newblock {\em The {N}avier-{S}tokes Equations}.
\newblock Birkh\"auser, Basel, 2001.

\bibitem{webercompmax}
C.~Weber.
\newblock A local compactness theorem for {M}axwell's equations.
\newblock {\em Math. Methods Appl. Sci.}, 2:12--25, 1980.

\bibitem{weberregmax}
C.~Weber.
\newblock Regularity theorems for {M}axwell's equations.
\newblock {\em Math. Methods Appl. Sci.}, 3:523--536, 1981.

\bibitem{weckmax}
N.~Weck.
\newblock {M}axwell's boundary value problems on {R}iemannian manifolds with
  nonsmooth boundaries.
\newblock {\em J. Math. Anal. Appl.}, 46:410--437, 1974.

\bibitem{wecktrace}
N.~Weck.
\newblock Traces of differential forms on {L}ipschitz boundaries.
\newblock {\em Analysis (Munich)}, 24:147--169, 2004.

\bibitem{witschremmax}
K.-J. Witsch.
\newblock A remark on a compactness result in electromagnetic theory.
\newblock {\em Math. Methods Appl. Sci.}, 16:123--129, 1993.

\end{thebibliography}


\end{document}